\documentclass[11pt]{amsart}

\usepackage{amsmath,amssymb,amsthm,mathtools}
\usepackage{enumitem}
\usepackage{iftex}
\ifptex
  \usepackage[dvipdfmx,hidelinks]{hyperref}
\else
  \usepackage[hidelinks]{hyperref}
\fi
\usepackage{cleveref}

\theoremstyle{plain}
\newtheorem{theorem}{Theorem}[section]

\newtheorem{lemma}[theorem]{Lemma}
\newtheorem{corollary}[theorem]{Corollary}
\crefname{theorem}{theorem}{theorems}
\Crefname{theorem}{Theorem}{Theorems}
\crefname{proposition}{proposition}{propositions}
\Crefname{proposition}{Proposition}{Propositions}
\crefname{lemma}{lemma}{lemmas}
\Crefname{lemma}{Lemma}{Lemmas}
\crefname{corollary}{corollary}{corollaries}
\Crefname{corollary}{Corollary}{Corollaries}

\theoremstyle{definition}
\newtheorem{definition}[theorem]{Definition}
\crefname{definition}{definition}{definitions}
\Crefname{definition}{Definition}{Definitions}

\theoremstyle{remark}

\crefname{remark}{remark}{remarks}
\Crefname{remark}{Remark}{Remarks}
\crefname{section}{section}{sections}
\Crefname{section}{Section}{Sections}
\crefname{subsection}{subsection}{subsections}
\Crefname{subsection}{Subsection}{Subsections}
\crefname{equation}{equation}{equations}
\Crefname{equation}{Equation}{Equations}

\setlist[enumerate]{leftmargin=2em,itemsep=0.25em,topsep=0.25em}
\newcommand{\R}{\mathbb R}
\newcommand{\D}{\mathcal D}
\newcommand{\DK}{\mathcal D_K}
\newcommand{\T}{\mathcal T}
\newcommand{\Kc}{\mathsf K}
\newcommand{\DM}{\mathcal{DM}}
\newcommand{\dconc}{d_{\mathrm{conc}}}
\newcommand{\dtau}{d_{\tau}}
\newcommand{\kf}{d_{\operatorname{KF}}}
\DeclareMathOperator{\pr}{pr}
\DeclareMathOperator{\supp}{supp}

\title[Compact Screens for Geometric Data Sets]{Compact Screens and Pyramidal Compactification of Geometric Data Sets}
\author{Shigeaki Yokota}
\thanks{Graduate School of Science, Tohoku University, Sendai 980-8578, Japan. Corresponding author: \texttt{shigeaki.yokota.t4@dc.tohoku.ac.jp}}
\address{Graduate School of Science, Tohoku University, Sendai 980-8578, Japan}
\email{shigeaki.yokota.t4@dc.tohoku.ac.jp}
\date{}

\begin{document}

\begin{abstract}
We introduce an observable distance that compares real-valued features through a fixed bounded coordinate. The coordinate retains the distinction between finite feature values while compressing their independent escape to infinity, the source of nonseparability for the classical observable distance on all geometric data sets. The new distance makes the full class separable and geodesic: every pair is joined by a constant-speed path, and on metric measure spaces the induced topology agrees with the concentration topology. From the same coordinate we construct compact screens, whose features take values in one fixed compact interval. The screened class is Polish and geodesic for the Box distance. Organizing its finite-feature quotients by the feature order yields a compact pyramid space. Pyramids generated by single compact screens form a dense subspace, so this pyramid space compactifies the original class after passage to compact screens. Convergence is detected through the Box-Hausdorff behavior of every finite measurement layer. Finally, taking sum-metric products with a common metric measure factor is nonexpansive for the new distance. More precisely, the comparison determined by a prescribed coupling of the original spaces and the diagonal coupling of the common factor is preserved, whereas optimization over all product couplings yields the nonexpansive inequality.
\end{abstract}

\maketitle

\section{Introduction}
\label{sec:introduction}

The comparison of metric probability spaces through real-valued observations provides a common language for measure concentration and degeneration of spaces \cite[Chapter~3]{gromov2007met}. A metric measure space, or mm-space, uses all real-valued $1$-Lipschitz functions as observations, and Gromov's observable distance compares the resulting families on parameter spaces \cite[Definition~3\(\frac12\).45, p.~199]{gromov2007met}. A geometric data set instead includes a family of real-valued features chosen for the problem at hand \cite[Definition~3.1]{hanika2022gds}. This flexibility retains only the prescribed observations, but direct comparison of their values on the real line leaves no countable approximation family for the full class. The resulting nonseparability also rules out an embedding into a compact metric space.

The distinction from mm-spaces lies in the relation between geometry and observations. For an mm-space, the convention of taking every $1$-Lipschitz function fixes the observational family through the metric. A geometric data set may choose its feature family separately while requiring the induced metric to remain complete and separable. This makes the choice of observations part of the object and distinguishes feature families carried by the same underlying space and measure. Compactness of each underlying space therefore gives no control over the topology of the full class.

A compactification must retain more than finite approximations of individual objects. Finite observations of different objects must be compared in one metric, and their containment relations must survive passage to the limit. In the classical theory, pyramids organize finite observations by the Lipschitz order \cite[Definitions~6.3--6.4]{shioya2016mmg}. The corresponding construction for all geometric data sets first requires a separable metric compatible with their prescribed features.

The obstruction already occurs on one-point spaces. For every nonempty set $A$ of positive integers, let $c_a$ be the constant function with value $a$ on a one-point space and set
\[X_A\coloneqq\bigl(\{*\},\{c_a\mid a\in A\},\delta_*\bigr).\]
Write $\dconc$ for the observable distance defined in \Cref{sec:preliminaries}. If $A\neq B$, then $\dconc(X_A,X_B)=1$ \cite[Theorem~1.1]{gds1}. Thus the isomorphism classes of geometric data sets contain an uncountable $1$-separated family. Every member has the same underlying probability space, and only the feature values produce the separation. The obstruction is the independent escape of observations along the real line, not the size of the underlying space.

On a one-point probability space, the Ky Fan metric \cite[Definition~1.23]{shioya2016mmg} between two constants is their absolute difference truncated at $1$. Distinct positive integers are therefore all at distance $1$, and the preceding calculation reduces to the Hausdorff distance between subsets of the integers. Neither geometric degeneration nor dispersion of mass occurs in this example. Refining the approximation of the underlying space cannot remove the obstruction, so the comparison of feature values itself must change.

We compare every feature through one fixed bounded monotone coordinate. More precisely, we use one half of the hyperbolic tangent for every geometric data set and measure the mean discrepancy after this transformation. The coordinate remains injective at finite values while compressing differences near infinity. For example, the transformed distance between consecutive positive integers tends to zero. This behavior differs from truncating every difference at a fixed threshold.

The bounded coordinate has three uses. Its range and Lipschitz constant are uniform over all features. Mean error is compatible with gluing couplings \cite[Definition~1.20]{shioya2016mmg} and with interpolation inside the bounded coordinate. Finally, completing the underlying space for the metric induced by transformed features gives a \emph{compact screen}, which allows finite observations to be compared by the Box distance \cite[Definition~5.7]{gds1}. These three constructions arise from the same coordinate rather than from separate truncations.

Let $\D$ denote the isomorphism classes of geometric data sets, and write $\dtau$ for the compactified observable distance defined in \Cref{sec:bounded-observable-distance}. Denote the compact screen of $X$ by $\mathsf C_\tau(X)$. Let $K=[-1/2,1/2]$, and let $\DK$ be the isomorphism classes whose features take values in $K$. We call $\DK$ the \emph{screened class}. \Cref{sec:pyramid-compactification} defines the space $\Pi_K$ of pyramids in $\DK$ and a metric $d_{\Pi,K}$ assembled from the Box Hausdorff distances of finite-feature measurements. For $Z\in\DK$, let $\mathcal P_Z$ be the \emph{principal pyramid} consisting of all objects dominated by $Z$. Precise definitions are given where these objects first enter the proofs.

\begin{theorem}
\label{thm:intro-compactification}
The metric space $(\Pi_K,d_{\Pi,K})$ is compact. The map
\[(\D,\dtau)\longrightarrow(\Pi_K,d_{\Pi,K}),\qquad X\longmapsto\mathcal P_{\mathsf C_\tau(X)}\]
is a topological embedding with dense image.
\end{theorem}

The compact space in \Cref{thm:intro-compactification} is obtained from finite observations in the same bounded coordinate: $\dtau$ makes $\D$ separable and geodesic, while $\DK$ is Polish and geodesic for the Box distance. A point of $\Pi_K$ is a nonempty Box-closed family in the screened class that is downward closed and directed for the feature order \cite[Definition~3.8]{gds1}. Principal pyramids are generated by single screened objects, and density means that every pyramid is approximated by such points.

The theorem compactifies $(\D,\dtau)$, not the full class equipped with $\dconc$, which contains the nonseparable one-point family above. On the mm-space subspace, $\dtau$ and $\dconc$ determine the same convergent sequences. Passing to compact screens places finite observations in the fixed interval $K$, and the resulting image is Box dense in $\DK$.

\Cref{sec:preliminaries} introduces the geometric data sets, couplings, distances, feature order, and quotients used below. \Cref{sec:bounded-observable-distance} establishes metricity, separability, geodesicity, agreement with the mm-space topology, and the strict limitation on the full class. Compact screens are constructed in \Cref{sec:compact-screen-box}, which proves their Box geometry, Polishness, comparison estimates, and density. \Cref{sec:pyramid-compactification} controls weak pyramid limits by finite measurements and proves \Cref{thm:intro-compactification}. \Cref{sec:tensorization} distinguishes fixed-coupling equality from optimized nonexpansiveness for sum-metric products with a common factor. \Cref{app:bounded-feature-inversion} contains the compactness and inversion arguments, \Cref{app:pyramid-technical} supplies the finite-measurement details, and \Cref{app:pooling} proves cyclic-pooling convergence and the finite-support approximation of a general factor.

\section{Preliminaries}
\label{sec:preliminaries}

\subsection{Geometric data sets and couplings}
\label{subsec:gds-coupling}

Geometric data sets and mm-spaces share the same metric-measure notation, and couplings provide their common probability realizations.

\begin{definition}[Geometric data set {\cite[Definition~3.1]{hanika2022gds}}]
\label{def:gds}
A triple $(X,F_X,\mu_X)$ is called a \emph{geometric data set} if $F_X$ is a nonempty family of real-valued functions on $X$, the formula
\[d_{F_X}(x,x')\coloneqq\sup_{f\in F_X}|f(x)-f(x')|\]
defines a complete separable metric on $X$, and $\mu_X$ is a Borel probability measure with full support for this metric. We also denote this metric by $d_X$.
\end{definition}

Every $f\in F_X$ is $1$-Lipschitz on $(X,d_X)$. We write $\overline{F_X}$ for the closure of $F_X$ in the topology of pointwise convergence.

\begin{definition}[Isomorphism {\cite[Definition~3.2]{hanika2022gds}}]
\label{def:gds-isomorphism}
Geometric data sets $X$ and $Y$ are \emph{isomorphic} if there is a Borel map $\phi\colon X\to Y$ such that
\[\phi_*\mu_X=\mu_Y,\qquad \overline{F_Y}\circ\phi=\overline{F_X}.\]
\end{definition}

We identify isomorphic geometric data sets throughout.

The notion of an mm-space goes back to Gromov \cite[Definition~3\(\frac12\).1, p.~113]{gromov2007met}. We use the probability and full-support convention of \cite[Definition~2.8]{shioya2016mmg}.

\begin{definition}[mm-space {\cite[Definition~2.8]{shioya2016mmg}}]
\label{def:mm-space}
A triple $(X,d_X,\mu_X)$ is called an \emph{mm-space} if $(X,d_X)$ is a complete separable metric space and $\mu_X$ is a Borel probability measure with full support.
\end{definition}

Let $\operatorname{Lip}_1(X,d_X)$ denote the real-valued $1$-Lipschitz functions on $(X,d_X)$. We regard an mm-space $(X,d_X,\mu_X)$ as the geometric data set
\[\bigl(X,\operatorname{Lip}_1(X,d_X),\mu_X\bigr).\]
This feature family induces $d_X$, and the identification makes mm-spaces a subclass of $\D$.

\begin{definition}[Coupling {\cite[Definition~1.20]{shioya2016mmg}}]
\label{def:coupling}
Let $\mu$ and $\nu$ be Borel probability measures on complete separable metric spaces. A \emph{coupling} of $\mu$ and $\nu$ is a Borel probability measure on the product space with marginals $\mu$ and $\nu$. We write $\T(\mu,\nu)$ for the set of all such couplings.
\end{definition}

Write $\pr_1$ and $\pr_2$ for the coordinate projections from a product, and write $\supp\nu$ for the support of a Borel measure $\nu$. For measurable real-valued functions $u,v$ on a probability space $(Z,\xi)$, set
\[d_{1,\xi}(u,v)\coloneqq\int_Z|u-v|\,d\xi.\]
The Hausdorff distance induced by a metric $d$ is denoted by $(d)_H$.

\subsection{Classical observable and Box distances}
\label{subsec:classical-distances}

The classical observable and Box distances provide the comparison interfaces for the bounded-coordinate constructions.

\begin{definition}[Ky Fan metric {\cite[Definition~1.23]{shioya2016mmg}}]
\label{def:ky-fan}
For measurable real-valued functions $u,v$ on a probability space $(Z,\xi)$, their \emph{Ky Fan metric} is
\[\kf^\xi(u,v)\coloneqq\inf\{\epsilon\geq0\mid\xi(\{|u-v|>\epsilon\})\leq\epsilon\}.\]
\end{definition}

\begin{definition}[Parameter {\cite[Definition~4.1]{gds1}}]
\label{def:parameter}
Set $I\coloneqq[0,1)$, and let $\lambda$ be Lebesgue measure. A Borel map $\varphi\colon I\to X$ is called a \emph{parameter} of $\mu_X$ if $\varphi_*\lambda=\mu_X$.
\end{definition}

The following definition extends Gromov's observable distance for mm-spaces \cite[Definition~3\(\frac12\).45, p.~199]{gromov2007met} to geometric data sets \cite[Definition~4.2]{gds1}.

\begin{definition}[Observable distance {\cite[Definition~4.2]{gds1}}]
\label{def:dconc}
For geometric data sets $X$ and $Y$, their \emph{observable distance} is
\[\dconc(X,Y)\coloneqq\inf_{\varphi,\psi}(\kf^\lambda)_H\bigl(F_X\circ\varphi,F_Y\circ\psi\bigr),\]
where $\varphi$ and $\psi$ range over the parameters of $\mu_X$ and $\mu_Y$, respectively.
\end{definition}

The observable distance has the following optimal-coupling representation \cite[Theorem~4.6]{gds1}:
\begin{equation}
\label{eq:dconc-coupling}
\dconc(X,Y)=\min_{\pi\in\T(\mu_X,\mu_Y)}(\kf^\pi)_H\bigl(F_X\circ\pr_1,F_Y\circ\pr_2\bigr).
\end{equation}
Moreover, $\dconc(X,Y)=0$ if and only if $X$ and $Y$ are isomorphic \cite[Theorem~3.10]{hanika2022gds}.

For real-valued functions $u,v$ on a product space and a closed set $S$, define
\[d_{\infty,S}(u,v)\coloneqq\sup_{z\in S}|u(z)-v(z)|\]
when $S\neq\varnothing$, and set $d_{\infty,\varnothing}\coloneqq0$.

The following distance extends Gromov's $\Box_\lambda$ construction for mm-spaces \cite[Definitions~3\(\frac12\).2--3\(\frac12\).3, pp.~116--117]{gromov2007met} to geometric data sets \cite[Definition~5.7]{gds1}.

\begin{definition}[Box distance {\cite[Definition~5.7]{gds1}}]
\label{def:box}
For geometric data sets $X$ and $Y$, their \emph{Box distance} is
\begin{equation}
\label{eq:box-formula}
\Box(X,Y)\coloneqq\inf_{\substack{\pi\in\T(\mu_X,\mu_Y)\\S\subset X\times Y\ \mathrm{closed}}}\max\left\{1-\pi(S),2(d_{\infty,S})_H\bigl(F_X\circ\pr_1,F_Y\circ\pr_2\bigr)\right\}.
\end{equation}
\end{definition}

Replacing either feature family by its pointwise closure does not change the Box distance, and the infimum in \Cref{eq:box-formula} is attained by a coupling and a closed set \cite[Theorem~5.14]{gds1}. The Box distance is a complete metric on $\D$ \cite[Theorem~1.4]{gds1}, and
\[\dconc(X,Y)\leq\Box(X,Y)\]
by \cite[Proposition~5.8]{gds1}.

\subsection{Feature order and quotients}
\label{subsec:feature-order-quotient}

The feature order records domination by prescribed observations, while feature quotients produce the finite objects used later.

The feature order \cite[Definition~3.8]{gds1} extends Gromov's Lipschitz order on mm-spaces \cite[Definition~3\(\frac12\).15, p.~134]{gromov2007met} to geometric data sets.

\begin{definition}[Feature order {\cite[Definition~3.8]{gds1}}]
\label{def:feature-order}
For geometric data sets $X$ and $Y$, we say that $X$ \emph{dominates} $Y$, and write $Y\preceq X$, if there is a Borel map $p\colon X\to Y$ such that
\[p_*\mu_X=\mu_Y,\qquad F_Y\circ p\subset\overline{F_X}.\]
This relation is called the \emph{feature order}.
\end{definition}

Every domination map is $1$-Lipschitz \cite[Proposition~3.10]{gds1}. We use feature quotients to construct the dominated objects needed below.

\begin{definition}[Feature quotient {\cite[Definition~3.16 and Propositions~3.17 and~3.19]{gds1}}]
\label{def:feature-quotient}
Let $G\subset\overline{F_X}$ be a nonempty pointwise-closed family, and set
\[d_G(x,x')\coloneqq\sup_{g\in G}|g(x)-g(x')|.\]
Complete the metric quotient of this pseudometric and restrict the completed space to the support of the pushforward of $\mu_X$ under the canonical map. Every $g\in G$ descends to the quotient and extends continuously to the completion. The resulting geometric data set is the \emph{feature quotient} of $X$ by $G$ and is denoted by $X/G$. It is unique up to isomorphism and satisfies
\[X/G\preceq X.\]
\end{definition}

For finite $G$, the quotient $X/G$ represents a finite observation in \Cref{sec:compact-screen-box,sec:pyramid-compactification}.

The construction in \Cref{sec:bounded-observable-distance} requires compactness of bounded Lipschitz families and a realization procedure for function-induced pseudometrics.

\begin{lemma}
\label{lem:bounded-lipschitz-compactness}
Let $(Z,d)$ be a complete separable metric space, let $\mu$ be a Borel probability measure on $Z$ with full support, and let $\mathcal A$ be a nonempty family of real-valued functions on $Z$. Suppose that there are $B,\ell\geq0$ such that
\[|u(z)|\leq B,\qquad |u(z)-u(z')|\leq\ell d(z,z')\]
for every $u\in\mathcal A$ and all $z,z'\in Z$. Then the $L^1(\mu)$-closure of $\mathcal A$ is compact, and each of its elements has a unique continuous $\ell$-Lipschitz representative. Under this identification, the pointwise closure and the $L^1(\mu)$-closure of $\mathcal A$ coincide.
\end{lemma}

Compactness of uniformly bounded Lipschitz families and the identification with pointwise convergence follow from \cite[Lemma~2.2 and Remark~2.3]{hanika2022gds}. The additional formulation in terms of continuous representatives and the $L^1$-closure is proved in \Cref{subsec:app-bounded-realization}.

\begin{lemma}
\label{lem:feature-realization}
Let $Z$ be a separable metrizable space, let $\nu$ be a Borel probability measure on $Z$, and let $H$ be a nonempty family of continuous real-valued functions on $Z$. Suppose that
\[d^H(z,z')\coloneqq\sup_{h\in H}|h(z)-h(z')|\]
is finite, its metric quotient $Z_H$ is separable, and the canonical map from $Z$ to $Z_H$ is Borel. Let $\widehat Z_H$ be the completion of $Z_H$, let $q\colon Z\to\widehat Z_H$ be the canonical map, and set
\[Z_H^\nu\coloneqq\supp(q_*\nu).\]
Then every $h\in H$ extends uniquely from the metric quotient to a continuous function $\widehat h$ on $\widehat Z_H$, and
\[\left(Z_H^\nu,\{\widehat h|_{Z_H^\nu}\mid h\in H\},q_*\nu|_{Z_H^\nu}\right)\]
is a geometric data set. If every member of $H$ takes values in a closed interval $J$, then every extended feature also takes values in $J$. If $\nu$ has full support and $q$ is continuous, then $q(Z)\subset Z_H^\nu$.
\end{lemma}

The quotient construction and its universal property for a geometric data set are given in \cite[Definition~3.16, Proposition~3.17, Claim~3.18, and Proposition~3.19]{gds1}. The extension needed when the domain is only assumed to be separable and metrizable is proved in \Cref{subsec:app-bounded-realization}.

\section{The Compactified Observable Distance}
\label{sec:bounded-observable-distance}

Constant features already make the full class nonseparable when real-valued features are compared without modification. We place every feature value in the same bounded coordinate and measure the resulting Hausdorff $L^1$ cost on couplings. This construction yields a metric that is separable and geodesic and that induces the classical topology on mm-spaces. A one-point example at the end of the section shows that the classical topology is not preserved on the full class.

\subsection{Metricity}
\label{subsec:dtau-metricity}

Define
\[
\tau\colon\R\longrightarrow(-1/2,1/2),\qquad
\tau(r)\coloneqq\frac12\tanh r.
\]
This is an odd, strictly increasing, $1/2$-Lipschitz homeomorphism. Its particular form is needed for more than boundedness. The addition formula for the hyperbolic tangent gives the concavity under nonnegative translations used in the pooling contraction in \Cref{sec:tensorization}.

\begin{definition}[Compactified feature set]
\label{def:compactified-feature-set}
For a geometric data set $X$, define its \emph{compactified feature set} by
\[
\Kc_\tau(X)\coloneqq
\overline{\{\tau\circ f\mid f\in F_X\}}^{\,\mathrm{pt}},
\]
where the closure is taken with respect to pointwise convergence on $X$.
\end{definition}

For a coupling $\pi\in\T(\mu_X,\mu_Y)$, put
\[
D_\tau^\pi(X,Y)\coloneqq(d_{1,\pi})_H
\bigl(\Kc_\tau(X)\circ\pr_1,\Kc_\tau(Y)\circ\pr_2\bigr).
\]

\begin{definition}[Compactified observable distance]
\label{def:dtau}
For geometric data sets $X$ and $Y$, define the \emph{compactified observable distance} by
\[
\dtau(X,Y)\coloneqq\inf_{\pi\in\T(\mu_X,\mu_Y)}D_\tau^\pi(X,Y).
\]
\end{definition}

\begin{lemma}
\label{lem:feature-compactness}
For every geometric data set $X$, the set $\Kc_\tau(X)$ is a nonempty compact subset of $L^1(\mu_X)$. Each of its elements has a unique continuous $1/2$-Lipschitz representative $u\colon X\to[-1/2,1/2]$. Moreover,
\[
\Kc_\tau(X)=\{\tau\circ f\mid f\in\overline{F_X}\}\cup E_X,
\]
where $E_X$ consists of the constant functions in $\Kc_\tau(X)$ with value $-1/2$ or $1/2$.
\end{lemma}

\begin{proof}
Apply \Cref{lem:bounded-lipschitz-compactness} to $\{\tau\circ f\mid f\in F_X\}$. This gives compactness, the unique continuous $1/2$-Lipschitz representatives, and the equality between the pointwise and $L^1(\mu_X)$ closures.

Take $u\in\Kc_\tau(X)$ and a sequence $f_n\in F_X$ such that $\tau\circ f_n\to u$ pointwise as $n\to\infty$. Suppose that $u(x_0)=1/2$ for some $x_0\in X$. Then $f_n(x_0)\to+\infty$. For every $x\in X$,
\[
f_n(x)\geq f_n(x_0)-d_X(x,x_0),
\]
and hence $u$ is identically $1/2$. The argument for the value $-1/2$ is the same.

Suppose that $u$ does not take either endpoint value. The function $f\coloneqq\tau^{-1}\circ u$ is real-valued, and $f_n\to f$ pointwise as $n\to\infty$. Thus $f\in\overline{F_X}$. Conversely, if $f\in\overline{F_X}$, apply the dominated convergence theorem \cite[Theorem~3.31, pp.~92--93]{axler2020mira} to a sequence in $F_X$ converging pointwise to $f$. This gives $\tau\circ f\in\Kc_\tau(X)$. This completes the proof.
\end{proof}

\begin{theorem}
\label{thm:dtau-metric}
The function $\dtau$ is a metric on $\D$, and the infimum over couplings in its definition is attained.
\end{theorem}

\begin{proof}
Weak compactness of the set of couplings \cite[Lemma~2.10]{gds1} and finite $L^1$-nets show that $\pi\mapsto D_\tau^\pi(X,Y)$ is continuous. Therefore, the infimum is attained. The details are given in \Cref{subsec:app-feature-coupling}.

We first check that the value is independent of the chosen isomorphic representative. Let $\phi\colon X\to X'$ be an isomorphism. The equality $\overline{F_{X'}}\circ\phi=\overline{F_X}$ shows that $\phi$ is an isometric embedding for the metrics induced by the feature families. Its image is closed by completeness and dense because $\phi_*\mu_X=\mu_{X'}$ and $\mu_{X'}$ has full support. Thus $\phi$ is surjective. For any $Y$, pushing a coupling of $X$ and $Y$ forward by $\phi\times\operatorname{id}_Y$ preserves the $L^1$ costs between compactified features. The inverse isomorphism gives the reverse correspondence, and therefore $\dtau(X,Y)=\dtau(X',Y)$.

Nonnegativity and symmetry follow from the definition. We prove the triangle inequality for $X,Y,W\in\D$. Apply the gluing lemma \cite[Lemma~7.6]{villani2009optimal} to the optimal couplings $\pi_{XY}$ and $\pi_{YW}$ over their common marginal $\mu_Y$, and denote the resulting probability measure on $X\times Y\times W$ by $\gamma$. Write $\pr_X,\pr_Y,\pr_W$ for the coordinate projections and put
\[
\pi_{XW}\coloneqq(\pr_X,\pr_W)_*\gamma.
\]
For $u\in\Kc_\tau(X)$ and $\epsilon>0$, choose successively $v\in\Kc_\tau(Y)$ and $w\in\Kc_\tau(W)$ such that
\[
\begin{aligned}
d_{1,\pi_{XY}}(u\circ\pr_X,v\circ\pr_Y)
&\leq D_\tau^{\pi_{XY}}(X,Y)+\epsilon,\\
d_{1,\pi_{YW}}(v\circ\pr_Y,w\circ\pr_W)
&\leq D_\tau^{\pi_{YW}}(Y,W)+\epsilon.
\end{aligned}
\]
The triangle inequality in $L^1(\gamma)$ gives
\[
d_{1,\pi_{XW}}(u\circ\pr_X,w\circ\pr_W)
\leq D_\tau^{\pi_{XY}}(X,Y)+D_\tau^{\pi_{YW}}(Y,W)+2\epsilon.
\]
The same argument beginning with a feature of $W$ gives the reverse directed estimate. Letting $\epsilon\downarrow0$ and taking the infimum over the $XW$-couplings yields
\[
\dtau(X,W)\leq\dtau(X,Y)+\dtau(Y,W).
\]

It remains to prove separation. Suppose that $\dtau(X,Y)=0$, let $\pi$ be an optimal coupling, and set $Z\coloneqq\supp\pi$. The two compactified feature sets pulled back to $Z$ agree in $L^1(Z,\pi)$. For $f\in F_X$, there is $v\in\Kc_\tau(Y)$ such that
\[
\tau\circ f\circ\pr_1=v\circ\pr_2
\]
holds $\pi$-almost everywhere. Both sides are continuous on $Z$, and $\pi$ has full support on $Z$, so the equality holds at every point of $Z$. The left-hand side takes values in $(-1/2,1/2)$, and hence $v$ is not an endpoint constant. By \Cref{lem:feature-compactness}, there is $g\in\overline{F_Y}$ such that $v=\tau\circ g$. Injectivity of $\tau$ gives
\[
f\circ\pr_1=g\circ\pr_2\qquad\text{on }Z.
\]
Interchanging $X$ and $Y$ gives the reverse matching. If a sequence in $F_Y$ converges pointwise to $g$, then its pullback to $Z$ converges $\pi$-almost everywhere and therefore in measure. It follows from \Cref{eq:dconc-coupling} that $\dconc(X,Y)=0$. Separation for the observable distance \cite[Theorem~3.10]{hanika2022gds} shows that $X$ and $Y$ are isomorphic. This completes the proof.
\end{proof}

\subsection{Separability and geodesicity}
\label{subsec:dtau-geometry}

Parameters place all compactified feature sets in the common separable space $L^1(I,\lambda)$. A geodesic is obtained by matching two compact feature sets on an optimal coupling and interpolating linearly in the $\tau$-coordinate.

\begin{theorem}
\label{thm:dtau-separable}
The metric space $(\D,\dtau)$ is separable.
\end{theorem}

\begin{proof}
For each $X\in\D$, choose a parameter $\phi_X\colon I\to X$, whose existence follows from \cite[Theorem~17.41]{kechris1995classical}, and put
\[
K_X\coloneqq
\overline{\{\tau\circ f\circ\phi_X\mid f\in F_X\}}^{\,L^1(\lambda)}.
\]
By \Cref{lem:feature-compactness}, the set $K_X$ is a nonempty compact subset of the separable metric space $L^1(I,\lambda)$. The family of nonempty compact subsets of a separable metric space is separable for the Hausdorff distance. In fact, the nonempty finite subsets of a fixed countable dense set form a dense family.

There is therefore a countable family $(X_n)$ such that $(K_{X_n})$ is dense in $\{K_X\mid X\in\D\}$. Using $(\phi_X,\phi_{X_n})_*\lambda$ as a coupling of $X$ and $X_n$, we obtain
\[
\dtau(X,X_n)\leq(d_{1,\lambda})_H(K_X,K_{X_n}).
\]
Thus $(X_n)$ is dense in $(\D,\dtau)$. This completes the proof.
\end{proof}

\begin{lemma}
\label{lem:scalar-interpolation}
For $0<t<1$, put
\[
M_t(a,b)\coloneqq\tau^{-1}((1-t)\tau(a)+t\tau(b)).
\]
Then
\[
|M_t(a,b)-M_t(a',b')|
\leq\frac{|a-a'|}{1-t}+\frac{|b-b'|}{t}.
\]
For each $c\in\{-1/2,1/2\}$, the map
\[
a\longmapsto\tau^{-1}((1-t)\tau(a)+tc)
\]
is $(1-t)^{-1}$-Lipschitz, and the map
\[
b\longmapsto\tau^{-1}((1-t)c+t\tau(b))
\]
is $t^{-1}$-Lipschitz.
\end{lemma}

The derivative estimates are proved in \Cref{subsec:app-interpolation-inversion}.

\begin{theorem}
\label{thm:dtau-geodesic}
The metric space $(\D,\dtau)$ is geodesic. More precisely, for every $X,Y\in\D$, there is a path $(X_t)_{t\in[0,1]}$ such that $X_0=X$, $X_1=Y$, and
\[
\dtau(X_s,X_t)=|s-t|\dtau(X,Y)
\]
for all $s,t\in[0,1]$.
\end{theorem}

\begin{proof}
Take an optimal coupling $\pi\in\T(\mu_X,\mu_Y)$, and put $h\coloneqq\dtau(X,Y)$ and $Z\coloneqq\supp\pi$. On $Z$, set
\[
A\coloneqq\Kc_\tau(X)\circ\pr_1,\qquad
B\coloneqq\Kc_\tau(Y)\circ\pr_2.
\]
The sets $A$ and $B$ are compact in $L^1(Z,\pi)$ and satisfy $(d_{1,\pi})_H(A,B)=h$.

Choose countable dense subsets of $A$ and $B$ consisting of elements obtained from real-valued features. Assign to each member a partner at distance at most $h$, and let $R_0\subset A\times B$ be the collection of all assigned pairs. Its closure $R$ is compact, both coordinate projections of $R$ are surjective, and
\[
d_{1,\pi}(u,v)\leq h\qquad((u,v)\in R).
\]
For every generating pair in $R_0$, at least one component is chosen from an actual real-valued feature.

Fix $t\in(0,1)$. For $(u,v)\in R_0$, the function $(1-t)u+tv$ takes values in $(-1/2,1/2)$. Let $H_t$ be the family on $Z$ consisting of
\[
\tau^{-1}((1-t)u+tv),
\]
and let $d_t$ be its induced pseudometric. By \Cref{lem:scalar-interpolation},
\[
d_t(z,z')\leq
\frac{d_X(\pr_1z,\pr_1z')}{1-t}
+\frac{d_Y(\pr_2z,\pr_2z')}{t}.
\]
Thus $d_t$ is finite, the canonical map to the metric quotient is continuous, and the quotient is separable. Applying \Cref{lem:feature-realization} to $(Z,\pi,H_t)$ gives a geometric data set $(Q_t,F_t,\mu_t)$. Denote its isomorphism class by $X_t$, and set $X_0\coloneqq X$ and $X_1\coloneqq Y$.

Define
\[
\Phi_t\colon R\longrightarrow L^1(Z,\pi),\qquad
\Phi_t(u,v)\coloneqq(1-t)u+tv,
\]
and write $C_t\coloneqq\Phi_t(R)$. The compactified feature set of $X_t$, pulled back to $Z$, is $C_t$. This closure identification is verified in \Cref{subsec:app-feature-coupling}. At the endpoints, surjectivity of the two projections of $R$ gives $C_0=A$ and $C_1=B$.

Let $0\leq s<t\leq1$. Use the coupling induced by the canonical maps from $Z$ to $X_s$ and $X_t$. The set
\[
\{((1-s)u+sv,(1-t)u+tv)\mid(u,v)\in R\}
\]
projects onto both $C_s$ and $C_t$. Therefore,
\[
\dtau(X_s,X_t)
\leq(t-s)\sup_{(u,v)\in R}d_{1,\pi}(u,v)
\leq(t-s)h.
\]
Applying this estimate to $(X,X_s)$, $(X_s,X_t)$, and $(X_t,Y)$ and then using the triangle inequality gives
\[
h\leq\dtau(X,X_s)+\dtau(X_s,X_t)+\dtau(X_t,Y)
\leq sh+(t-s)h+(1-t)h=h.
\]
All the upper bounds are therefore equalities. This completes the proof.
\end{proof}

\subsection{The topology on mm-spaces}
\label{subsec:mm-topology}

The inverse of $\tau$ cannot be controlled uniformly near the endpoints. Since the full family of $1$-Lipschitz functions on an mm-space is invariant under the addition of constants, we normalize functions by a median and use tightness on the fixed space to prevent escape to the endpoints. A \emph{median} of a real-valued measurable function $f$ is a number $m\in\R$ such that
\[
\mu(\{f\leq m\})\geq\frac12,\qquad
\mu(\{f\geq m\})\geq\frac12.
\]

\begin{lemma}
\label{lem:inversion-tightness}
For each $n$, let $(W_n,\xi_n)$ be a probability space and let $a_n,b_n$ be real-valued measurable functions on $W_n$. Suppose that for every $\eta>0$, there are $M>0$ and $N\in\mathbb N$ such that
\[
\xi_n(\{|a_n|>M\})<\eta
\]
for every $n\geq N$. If
\[
d_{1,\xi_n}(\tau\circ a_n,\tau\circ b_n)\longrightarrow0
\qquad(n\to\infty),
\]
then
\[
\kf^{\xi_n}(a_n,b_n)\longrightarrow0
\qquad(n\to\infty).
\]
\end{lemma}

The proof in \Cref{subsec:app-interpolation-inversion} combines uniform continuity away from the endpoints with Markov's inequality.

\begin{theorem}
\label{thm:mm-topologies}
The topologies induced by $\dtau$ and $\dconc$ coincide on the subspace of mm-spaces. Equivalently, for mm-spaces $X$ and $X_n$,
\[
\dtau(X_n,X)\longrightarrow0
\quad\Longleftrightarrow\quad
\dconc(X_n,X)\longrightarrow0
\qquad(n\to\infty).
\]
\end{theorem}

\begin{proof}
Suppose that $\dconc(X_n,X)\to0$ as $n\to\infty$. Consider two functions whose Ky Fan distance on a common coupling is less than $\epsilon$. On the set where their difference is at most $\epsilon$, use the $1/2$-Lipschitz property of $\tau$. On the complement, use the diameter $1$ of its range. The $L^1$ distance between the transformed functions is at most $3\epsilon/2$, and the same estimate passes to the $L^1$ closures. Applying it to both directed feature approximations in \Cref{eq:dconc-coupling} gives $\dtau(X_n,X)\to0$ as $n\to\infty$.

Conversely, suppose that $\dtau(X_n,X)\to0$ as $n\to\infty$, and take optimal couplings $\pi_n\in\T(\mu_X,\mu_{X_n})$. It remains to prove
\begin{equation}
\label{eq:mm-kyfan}
(\kf^{\pi_n})_H\bigl(
\operatorname{Lip}_1(X,d_X)\circ\pr_1,
\operatorname{Lip}_1(X_n,d_{X_n})\circ\pr_2
\bigr)\longrightarrow0
\qquad(n\to\infty).
\end{equation}

Assume that \Cref{eq:mm-kyfan} fails. After passing to a subsequence, one of the two directed approximations fails by at least some $\epsilon>0$. In the first case, there are $f_n\in\operatorname{Lip}_1(X,d_X)$ such that
\[
\inf_{g\in\operatorname{Lip}_1(X_n,d_{X_n})}
\kf^{\pi_n}(f_n\circ\pr_1,g\circ\pr_2)\geq\epsilon.
\]
The addition of constants preserves the full Lipschitz family, so we may assume that $0$ is a median of $f_n$. Fix $x_0\in X$, and choose $R>0$ such that $\mu_X(B_R(x_0))>1/2$. This ball meets both $\{f_n\leq0\}$ and $\{f_n\geq0\}$. The $1$-Lipschitz property gives $|f_n(x_0)|\leq R$. By tightness of probability measures \cite[Definition~1.18 and Theorem~1.19]{shioya2016mmg}, for each $\eta>0$ there is a compact set $L\subset X$ such that $\mu_X(L)>1-\eta$. On $L$,
\[
|f_n(x)|\leq R+\sup_{z\in L}d_X(z,x_0).
\]
Thus $f_n\circ\pr_1$ satisfies the tail assumption in \Cref{lem:inversion-tightness}.

Choose $v_n\in\Kc_\tau(X_n)$ such that
\[
d_{1,\pi_n}(\tau\circ f_n\circ\pr_1,v_n\circ\pr_2)
\leq D_\tau^{\pi_n}(X,X_n).
\]
If $v_n$ is the endpoint constant $1/2$, the integrand is at least $1/2$ on $\{f_n\leq0\}$, so the integral is at least $1/4$. The endpoint constant $-1/2$ is excluded in the same way by using $\{f_n\geq0\}$. For every sufficiently large $n$, \Cref{lem:feature-compactness} therefore gives $g_n\in\operatorname{Lip}_1(X_n,d_{X_n})$ such that $v_n=\tau\circ g_n$. Applying \Cref{lem:inversion-tightness} yields
\[
\kf^{\pi_n}(f_n\circ\pr_1,g_n\circ\pr_2)\longrightarrow0
\qquad(n\to\infty),
\]
which contradicts the lower bound.

In the second case, there are $g_n\in\operatorname{Lip}_1(X_n,d_{X_n})$ such that
\[
\inf_{f\in\operatorname{Lip}_1(X,d_X)}
\kf^{\pi_n}(g_n\circ\pr_2,f\circ\pr_1)\geq\epsilon.
\]
We may assume that $0$ is a median of $g_n$. Choose $v_n\in\Kc_\tau(X)$ such that
\[
d_{1,\pi_n}(v_n\circ\pr_1,\tau\circ g_n\circ\pr_2)
\leq D_\tau^{\pi_n}(X,X_n).
\]
If $v_n$ is the endpoint constant $1/2$, the integrand is at least $1/2$ on $\{g_n\leq0\}$, and the integral is at least $1/4$. The endpoint constant $-1/2$ has the same lower bound on $\{g_n\geq0\}$. Thus, for every sufficiently large $n$, there is $f_n\in\operatorname{Lip}_1(X,d_X)$ such that $v_n=\tau\circ f_n$.

Set
\[
E_n\coloneqq
\{(x,y)\mid|\tau(f_n(x))-\tau(g_n(y))|\geq1/4\}.
\]
Then
\[
\pi_n(E_n)\leq4D_\tau^{\pi_n}(X,X_n)\longrightarrow0
\qquad(n\to\infty).
\]
Outside $E_n$, the inequality $g_n\leq0$ implies $f_n<(\log3)/2$, while $g_n\geq0$ implies $f_n>-(\log3)/2$. Therefore,
\[
\begin{aligned}
\mu_X(\{f_n<(\log3)/2\})&\geq\frac12-\pi_n(E_n),\\
\mu_X(\{f_n>-(\log3)/2\})&\geq\frac12-\pi_n(E_n).
\end{aligned}
\]
Fix a ball $B_R(x_0)$ in $X$ with measure greater than $3/4$. For every sufficiently large $n$, this ball meets both level sets. The $1$-Lipschitz property gives
\[
|f_n(x_0)|<R+\frac{\log3}{2}.
\]
Thus $(f_n)$ is uniformly bounded on every compact subset of $X$, and $f_n\circ\pr_1$ satisfies the tail assumption in \Cref{lem:inversion-tightness}. Applying the lemma gives a contradiction in the second case as well.

Neither directed approximation can fail, so \Cref{eq:mm-kyfan} holds. By \Cref{eq:dconc-coupling}, we have $\dconc(X_n,X)\to0$ as $n\to\infty$. This completes the proof.
\end{proof}

\subsection{Comparison on the full class}
\label{subsec:full-comparison}

The preceding normalization works on mm-spaces because adding constants preserves the full Lipschitz family. A general geometric data set need not have this invariance. The classical observable distance still controls $\dtau$ in one direction, but a uniform reverse estimate already fails for one-point spaces.

\begin{theorem}
\label{thm:comparison}
For all $X,Y\in\D$,
\[
0\leq\dtau(X,Y)
\leq\frac{3\dconc(X,Y)-\dconc(X,Y)^2}{2}
\leq1.
\]
On the other hand, there is no function $\omega\colon[0,1]\to[0,+\infty)$ such that $\omega(r)\to0$ as $r\to0$ and
\[
\dconc(X,Y)\leq\omega(\dtau(X,Y))
\]
for all $X,Y\in\D$.
\end{theorem}

\begin{proof}
Let $u,v$ be real-valued functions on a probability space $(W,\xi)$, and put $a\coloneqq\kf^\xi(u,v)$. Suppose first that $a<1$, and take $a<\epsilon<1$. For $B_\epsilon\coloneqq\{|u-v|>\epsilon\}$, we have $\xi(B_\epsilon)\leq\epsilon$. Using the $1/2$-Lipschitz property of $\tau$ outside $B_\epsilon$ and the diameter $1$ of its range on $B_\epsilon$, we obtain
\[
d_{1,\xi}(\tau\circ u,\tau\circ v)
\leq\frac{\epsilon}{2}(1-\xi(B_\epsilon))+\xi(B_\epsilon)
\leq\frac{3\epsilon-\epsilon^2}{2}.
\]
Letting $\epsilon\downarrow a$ gives the estimate. When $a=1$, it follows directly from the range diameter. The function $a\mapsto(3a-a^2)/2$ is increasing on $[0,1]$. Applying the estimate to both directed feature approximations on the optimal coupling in \Cref{eq:dconc-coupling} proves the upper bound.

For the failure of a reverse estimate, let $f_n(*)=n$ and $g_n(*)=2n$ on a one-point space, and put
\[
X_n\coloneqq(\{*\},\{f_n\},\delta_*),\qquad
Y_n\coloneqq(\{*\},\{g_n\},\delta_*).
\]
The coupling is unique, and
\[
\begin{aligned}
\dconc(X_n,Y_n)&=1,\\
\dtau(X_n,Y_n)
&=\frac12\bigl(\tanh(2n)-\tanh n\bigr)\\
&=\frac{e^{2n}(e^{2n}-1)}{(e^{4n}+1)(e^{2n}+1)}
\longrightarrow0
\qquad(n\to\infty).
\end{aligned}
\]
This contradicts the asserted inequality for every candidate $\omega$. This completes the proof.
\end{proof}

\begin{corollary}
\label{cor:full-topologies}
The identity map $(\D,\dconc)\to(\D,\dtau)$ is uniformly continuous, and the $\dconc$ topology is strictly finer than the $\dtau$ topology. Moreover, there is a sequence that converges with respect to $\dtau$ but has no $\dconc$-convergent subsequence.
\end{corollary}

\begin{proof}
Uniform continuity follows from \Cref{thm:comparison}. Let $c_r$ denote the function with value $r$ on a one-point space, and put
\[
F_\infty\coloneqq\{c_k\mid k=1,2,\ldots\},\qquad
Z_\infty\coloneqq(\{*\},F_\infty,\delta_*).
\]
For $n\geq1$, set
\[
Z_n\coloneqq(\{*\},F_\infty\cup\{c_{n+1/2}\},\delta_*).
\]
Then
\[
\dtau(Z_n,Z_\infty)
\leq\tau\left(n+\frac12\right)-\tau(n)
=\frac{e^{2n}(e-1)}{(e^{2n+1}+1)(e^{2n}+1)}
\longrightarrow0
\qquad(n\to\infty).
\]
The Ky Fan distance on a one-point probability space is $\min\{|r-s|,1\}$. Therefore,
\[
\dconc(Z_n,Z_\infty)=\frac12,
\qquad
\dconc(Z_n,Z_m)=\frac12\quad(n\neq m).
\]
Thus $(Z_n)$ has no $\dconc$-convergent subsequence. This completes the proof.
\end{proof}

This example shows that the bounded coordinate preserves the classical mm-space topology while compressing feature values at infinity on the full class. In \Cref{sec:compact-screen-box}, we incorporate the metric induced by the same coordinate into the objects and construct compact screens whose finite observations can be studied with the Box distance.

\section{Compact Screens and Box Geometry}
\label{sec:compact-screen-box}

The compactified observable distance compares transformed features on couplings without changing the metric of either object. We instead complete the support for the metric induced by the transformed features and thereby obtain objects whose features take values in one fixed interval. This construction yields Box geodesics, compact finite-feature layers, a Polish class, and direct comparisons with the original metric.

\subsection{Construction of compact screens}
\label{subsec:compact-screen-construction}

Let $K=[-1/2,1/2]$ be the interval fixed in the introduction. For $X\in\D$, put
\[
d_X^\tau(x,x')\coloneqq
\sup_{f\in F_X}|\tau(f(x))-\tau(f(x'))|.
\]
Denote the completion of $(X,d_X^\tau)$ by $\widehat X^\tau$ and the canonical map by $e_X^\tau\colon X\to\widehat X^\tau$. Every $\tau\circ f$ extends uniquely to a continuous function $\widehat{\tau\circ f}$ on $\widehat X^\tau$. Set
\[
X^\tau\coloneqq\supp((e_X^\tau)_*\mu_X).
\]

\begin{definition}[Compact screen]
\label{def:compact-screen}
The compact screen of a geometric data set $X$ is
\[
\mathsf C_\tau(X)\coloneqq
\left(
X^\tau,
\{\widehat{\tau\circ f}|_{X^\tau}\mid f\in F_X\},
(e_X^\tau)_*\mu_X|_{X^\tau}
\right).
\]
\end{definition}

\begin{lemma}
\label{lem:compact-screen-gds}
For every $X\in\D$, the compact screen $\mathsf C_\tau(X)$ is a geometric data set, and all its features take values in $K$. The construction is well-defined on isomorphism classes.
\end{lemma}

\begin{proof}
Since $d_X^\tau\leq d_X/2$, the map $e_X^\tau$ is continuous, and its image and metric completion are separable. Apply \Cref{lem:feature-realization} to
\[
(X,\mu_X,\{\tau\circ f\mid f\in F_X\}).
\]
This gives the geometric data set in \Cref{def:compact-screen} and shows that every feature takes values in $K$. Since $\mu_X$ has full support and $e_X^\tau$ is continuous, \Cref{lem:feature-realization} also gives $e_X^\tau(X)\subset X^\tau$. The left-hand side is dense in $\widehat X^\tau$, and therefore $X^\tau=\widehat X^\tau$.

An isomorphism from $X$ to $Y$ is an isometry with respect to $d_X^\tau$ and $d_Y^\tau$. It extends uniquely to a surjective isometry between the completions, and this extension preserves the feature families, pushforward measures, and their supports. This completes the proof.
\end{proof}

If $Z\in\DK$, then $\overline{F_Z}$ also consists of $K$-valued functions because pointwise limits remain in the closed interval $K$.

\begin{lemma}
\label{lem:K-feature-compactness}
If $Z\in\DK$, then $\overline{F_Z}$ is compact with respect to both the $L^1(\mu_Z)$ metric and the Ky Fan metric.
\end{lemma}

\begin{proof}
Compactness for the Ky Fan metric follows from \cite[Lemma~2.2 and Remark~2.3]{hanika2022gds}. For $K$-valued functions $f$ and $g$,
\[
d_{1,\mu_Z}(f,g)\leq2\kf^{\mu_Z}(f,g),\qquad
\kf^{\mu_Z}(f,g)\leq\sqrt{d_{1,\mu_Z}(f,g)}.
\]
The two metrics induce the same topology on $\overline{F_Z}$, which proves compactness for $L^1(\mu_Z)$. This completes the proof.
\end{proof}

\subsection{Box geodesics}
\label{subsec:box-geodesics}

Constructing a geodesic from an optimal realization of the Box distance requires a best matching feature on a closed relation. The following attainment result allows us to use the correspondence consisting of all optimal pairs.

\begin{lemma}
\label{lem:box-match-attainment}
Let $X,Y\in\D$, and let $S\subset X\times Y$ be a nonempty closed set. For every $f\in\overline{F_X}$, there is $g\in\overline{F_Y}$ such that
\[
d_{\infty,S}(f\circ\pr_1,g\circ\pr_2)
=\inf_{g'\in\overline{F_Y}}
d_{\infty,S}(f\circ\pr_1,g'\circ\pr_2).
\]
The analogous assertion holds after interchanging $X$ and $Y$.
\end{lemma}

This is the constant-sequence case $S_n=S$ of \cite[Lemma~5.18]{gds1}.

\begin{theorem}
\label{thm:box-geodesic}
The metric space $(\DK,\Box)$ is geodesic, and its geodesics may be taken in $\DK$. More generally, if $X,Y\in\D$ and $\Box(X,Y)<1$, then there is a constant-speed Box geodesic from $X$ to $Y$ in $\D$.
\end{theorem}

\begin{proof}
Put $h\coloneqq\Box(X,Y)$. If $h=0$, take the constant path. Suppose that $0<h<1$, and take a coupling $\pi$ and a closed set $S$ attaining \Cref{eq:box-formula}. Since the cost of the empty set is $1$, the set $S$ is nonempty. Put
\[
\varepsilon\coloneqq1-\pi(S),\qquad
\delta\coloneqq(d_{\infty,S})_H
(\overline{F_X}\circ\pr_1,\overline{F_Y}\circ\pr_2).
\]
Then $\varepsilon\leq h$ and $2\delta\leq h$. By \Cref{lem:box-match-attainment}, both coordinate projections of
\[
\mathcal R\coloneqq
\{(f,g)\in\overline{F_X}\times\overline{F_Y}\mid
d_{\infty,S}(f\circ\pr_1,g\circ\pr_2)\leq\delta\}
\]
are surjective.

Consider the topological disjoint union $E\coloneqq S\sqcup X\sqcup Y$. Let $\pi_{\mathrm g}$ be the restriction of $\pi$ to $S$, let $\pi_{\mathrm b}$ be its restriction to the complement of $S$, and denote the marginals of $\pi_{\mathrm b}$ by $\mu_X^{\mathrm b}$ and $\mu_Y^{\mathrm b}$. For $t\in[0,1]$, put
\[
\nu_t\coloneqq\pi_{\mathrm g}+(1-t)\mu_X^{\mathrm b}+t\mu_Y^{\mathrm b}.
\]
For $(f,g)\in\mathcal R$, define a function on $E$ by
\[
k_t^{f,g}(z)\coloneqq
\begin{cases}
(1-t)f(x)+tg(y),&z=(x,y)\in S,\\
f(x),&z=x\in X,\\
g(y),&z=y\in Y.
\end{cases}
\]
Let $d_t$ be the pseudometric induced by these functions. On $S$, it is bounded above by $(1-t)d_X+td_Y$, and on the endpoint components it is bounded above by $d_X$ and $d_Y$, respectively. Distances between $S$ and either endpoint component are finite and controlled by the endpoint metrics and $\delta$. More explicitly, fix $(x_0,y_0)\in S$. For $x$ in the $X$-component and $y$ in the $Y$-component,
\[
d_t(x,y)\leq d_X(x,x_0)+\delta+d_Y(y_0,y).
\]
Thus $d_t$ is finite and its metric quotient is separable. Apply \Cref{lem:feature-realization} to $(E,\nu_t,\{k_t^{f,g}\mid(f,g)\in\mathcal R\})$ to obtain a geometric data set $X_t$. Surjectivity of both projections of $\mathcal R$ shows that $X_0$ and $X_1$ are isomorphic to $X$ and $Y$, respectively.

Let $0\leq s<t\leq1$. Combine the diagonal mass on $S$, mass $1-t$ on the $X$-component, mass $s$ on the $Y$-component, and mass $t-s$ between the endpoint components induced by $\pi_{\mathrm b}$. This gives a coupling of $\nu_s$ and $\nu_t$. The closures of the three diagonal components have measure at least $1-(t-s)\varepsilon$. Features corresponding to the same $(f,g)$ have uniform distance at most $(t-s)\delta$ there. The common Lipschitz estimate and a diagonal subsequence preserve this bound under pointwise closure. Therefore,
\[
\Box(X_s,X_t)
\leq\max\{(t-s)\varepsilon,2(t-s)\delta\}
\leq(t-s)h.
\]
Apply this estimate to the three intervals with endpoints $X,X_s,X_t,Y$. The triangle inequality gives
\[
h\leq\Box(X,X_s)+\Box(X_s,X_t)+\Box(X_t,Y)
\leq sh+(t-s)h+(1-t)h=h.
\]
All the upper bounds are equalities. If $X,Y\in\DK$, then the convexity of $K$ implies that every feature constructed above is $K$-valued, so the path lies in $\DK$.

It remains to consider $X,Y\in\DK$ with $h=1$. Equip the disjoint union $X\sqcup Y$ with the measure $(1-t)\mu_X+t\mu_Y$. For every $(f,g)\in\overline{F_X}\times\overline{F_Y}$, take the function equal to $f$ on the $X$-component and to $g$ on the $Y$-component. Since $K$ has diameter $1$, these functions induce a finite separable pseudometric. Taking its metric quotient, completion, and support gives $X_t\in\DK$. Couple total diagonal mass $1-(t-s)$ on the two components and use a product coupling between the components for the remaining mass. Then
\[
\Box(X_s,X_t)\leq t-s.
\]
The equality $\Box(X,Y)=1$ and the triangle inequality again force equality. This completes the proof.
\end{proof}

\subsection{Finite-feature layers and Polishness}
\label{subsec:screened-polish}

When the number of features is fixed, each object can be represented by a probability measure on a compact cube. Compactness of these finite-dimensional layers combines with approximation by finite-feature quotients to give separability.

\begin{definition}[Finite-feature layer]
\label{def:bounded-measurements}
For $N\in\mathbb N$, let $\DM_K(N)$ be the set of all elements of $\DK$ with at most $N$ features. This set is called the \emph{finite-feature layer} of order $N$.
\end{definition}

\begin{lemma}
\label{lem:bounded-measurement-compact}
For every $N\in\mathbb N$, the set $\DM_K(N)$ is Box compact.
\end{lemma}

This is the case $R=1/2$ of \cite[Lemma~4.5]{gds2}.

\begin{theorem}
\label{thm:screened-polish}
The metric space $(\DK,\Box)$ is complete and separable.
\end{theorem}

\begin{proof}
By \cite[Theorem~4.10]{gds2}, the compact class obtained from the identity maps is Box complete. Compactness of $K$-valued features in the Ky Fan metric places $\DK$ in this class. In \Cref{subsec:app-screened-closed-density}, we prove that $\DK$ is closed in it. Thus $\DK$ is complete.

The same subsection shows that every $Z\in\DK$ can be approximated arbitrarily well by finite-feature quotients. Therefore, $\bigcup_{N\geq1}\DM_K(N)$ is dense in $\DK$. Each layer is compact and separable by \Cref{lem:bounded-measurement-compact}, so the countable union of dense subsets of these layers gives a countable dense subset of $\DK$. This completes the proof.
\end{proof}

\subsection{Metric comparison and density of compact screens}
\label{subsec:screen-comparison-density}

The Box distance after compact screening controls both the measure of an exceptional set and the uniform feature error outside it. Integrating these two contributions gives a direct comparison with $\dtau$. A finite example shows that no uniform reverse comparison exists, while rescaling the screened features proves density of compact screens.

\begin{definition}[Pullback Box metric]
\label{def:pullback-box}
For $X,Y\in\D$, define the \emph{pullback Box metric} by
\[
\Box_\tau(X,Y)\coloneqq
\Box(\mathsf C_\tau(X),\mathsf C_\tau(Y)).
\]
\end{definition}

\begin{corollary}
\label{cor:tau-box-comparison}
For all $X,Y\in\D$,
\[
\dtau(X,Y)
\leq\frac{3\Box_\tau(X,Y)-\Box_\tau(X,Y)^2}{2}.
\]
In particular, $\Box_\tau$ is a metric on $\D$. On the other hand, there is no function $\omega\colon[0,1]\to[0,+\infty)$ such that $\omega(r)\to0$ as $r\to0$ and
\[
\Box_\tau(X,Y)\leq\omega(\dtau(X,Y))
\]
for all $X,Y\in\DK$.
\end{corollary}

\begin{proof}
Put $h\coloneqq\Box_\tau(X,Y)$. If $h=1$, the diameter of the range of $\tau$ gives $\dtau(X,Y)\leq1$. Suppose that $h<1$, and take a coupling $\pi$ and a nonempty closed set $S$ attaining the Box distance between the two compact screens. Put
\[
\begin{aligned}
m&\coloneqq1-\pi(S),\\
\delta&\coloneqq(d_{\infty,S})_H
\bigl(\overline{F_{\mathsf C_\tau(X)}}\circ\pr_1,
\overline{F_{\mathsf C_\tau(Y)}}\circ\pr_2\bigr).
\end{aligned}
\]
Then $m\leq h$ and $\delta\leq h/2$. The canonical maps of the compact screens identify full-measure Borel subsets of the source and target as measure spaces, and the two feature closures correspond to $\Kc_\tau(X)$ and $\Kc_\tau(Y)$. By \Cref{lem:box-match-attainment}, every feature can be matched in both directions with a feature at uniform distance at most $\delta$ on $S$. Every matched pair $u,v$ satisfies
\[
d_{1,\pi}(u,v)
\leq\delta(1-m)+m
\leq\frac h2(1-h)+h
=\frac{3h-h^2}{2}.
\]
Taking the two directed Hausdorff distances and then the infimum over couplings proves the first estimate. If $\Box_\tau(X,Y)=0$, then $\dtau(X,Y)=0$, and \Cref{thm:dtau-metric} gives $X=Y$. The other metric axioms follow from the Box distance.

We prove that no uniform reverse estimate exists. Let $n\geq2$, let $E_n\coloneqq\{1,\ldots,n\}$ carry the uniform probability measure $\mu_n$, and define
\[
f_{n,i}(j)\coloneqq
\begin{cases}
1/4,&j=i,\\
0,&j\neq i.
\end{cases}
\]
Put
\[
X_n\coloneqq(E_n,\{f_{n,1},\ldots,f_{n,n}\},\mu_n),\qquad
Y\coloneqq(\{*\},\{0\},\delta_*).
\]
The feature distance between distinct points is $1/4$, and $X_n,Y\in\DK$. The coupling is unique, and
\[
\dtau(X_n,Y)=\frac{\tau(1/4)}{n}\longrightarrow0
\qquad(n\to\infty).
\]
Every nonempty $S\subset E_n\times\{*\}$ contains some $(i,*)$. The screened feature supported at $i$ has value $\tau(1/4)$ there, so its uniform distance from the zero feature on $S$ is $\tau(1/4)$. The full product space attains this value, and $2\tau(1/4)<1$. Therefore,
\[
\Box_\tau(X_n,Y)=2\tau(1/4)
\]
for every $n$. This contradicts every candidate $\omega$. This completes the proof.
\end{proof}

\begin{definition}[Screened observable distance]
\label{def:screened-observable-distance}
For $X,Y\in\D$, define the \emph{screened observable distance} by
\[
\dconc^\tau(X,Y)\coloneqq
\dconc(\mathsf C_\tau(X),\mathsf C_\tau(Y)).
\]
\end{definition}

\begin{lemma}
\label{lem:screened-topology}
For all $X,Y\in\D$,
\[
\dconc^\tau(X,Y)\leq\sqrt{\dtau(X,Y)},\qquad
\dtau(X,Y)\leq2\dconc^\tau(X,Y).
\]
\end{lemma}

\begin{proof}
Under the identification of measure spaces induced by the canonical maps of the compact screens, the pointwise closure of the screened feature family corresponds to $\Kc_\tau(X)$. Apply the two estimates in \Cref{lem:K-feature-compactness} to the directed Hausdorff distances on each coupling. This completes the proof.
\end{proof}

\begin{theorem}
\label{thm:screen-density}
The set $\mathsf C_\tau(\D)$ is dense in $(\DK,\Box)$. Moreover, $(\DK,\Box)$ is the metric completion of $(\D,\Box_\tau)$.
\end{theorem}

\begin{proof}
Take $Z=(Z,F_Z,\mu_Z)\in\DK$ and $0<r<1$. Put
\[
X_r\coloneqq
\left(Z,\{\tau^{-1}\circ(rf)\mid f\in F_Z\},\mu_Z\right).
\]
The inverse $\tau^{-1}$ is Lipschitz on the compact interval $rK$, so the feature metric of $X_r$ is bounded above by a constant multiple of $d_Z$. The $1/2$-Lipschitz property of $\tau$ gives
\[
2r d_Z\leq d_{X_r}.
\]
The two metrics are bi-Lipschitz equivalent, and $X_r$ is a geometric data set.

The screened feature family of $X_r$ is $rF_Z$, and its induced metric is $rd_Z$. Thus
\[
\mathsf C_\tau(X_r)=(Z,rF_Z,\mu_Z).
\]
Using the diagonal coupling and the full diagonal set in \Cref{eq:box-formula}, we obtain
\[
\Box(\mathsf C_\tau(X_r),Z)\leq1-r.
\]
Letting $r\uparrow1$ proves density. The map $\mathsf C_\tau$ is injective by \Cref{cor:tau-box-comparison} and is an isometric embedding by \Cref{def:pullback-box}. Completeness from \Cref{thm:screened-polish} proves the final assertion. This completes the proof.
\end{proof}

The space $\DK$, which contains the image of the compact-screen construction, is therefore complete, separable, and geodesic for the Box distance. \Cref{sec:pyramid-compactification} records all finite-feature quotients dominated by each $Z\in\DK$ and organizes their limits as pyramids.

\section{Pyramidal Compactification}
\label{sec:pyramid-compactification}

Approximating finite-feature quotients separately does not preserve which observations dominate which others in the limit. We organize the finite observations into pyramids under the feature order. Compactness of every finite-feature layer then controls weak limits and realizes the principal pyramids as a dense subspace of a compact metric space.

\subsection{Pyramids and finite measurements}
\label{subsec:pyramid-measurements}

We transfer the classical pyramid construction for mm-spaces to the screened class with fixed range $K$. A pyramid records not only its finite quotients but also downward closedness and the existence of common upper bounds.

\begin{definition}[$K$-pyramid {\cite[Definitions~6.3--6.4]{shioya2016mmg}}]
\label{def:K-pyramid}
A nonempty Box-closed subset $\mathcal P\subset\DK$ is called a \emph{$K$-pyramid} if it satisfies the following conditions.
\begin{enumerate}
\item If $X\in\mathcal P$, $Y\in\DK$, and $Y\preceq X$, then $Y\in\mathcal P$.
\item If $X,Y\in\mathcal P$, then there is $Z\in\mathcal P$ such that $X\preceq Z$ and $Y\preceq Z$.
\end{enumerate}
\end{definition}

We denote the set of all $K$-pyramids by $\Pi_K$.

\begin{definition}[Principal pyramid]
\label{def:principal-pyramid}
For $Z\in\DK$, the \emph{principal pyramid} generated by $Z$ is defined by
\[
\mathcal P_Z\coloneqq\{Y\in\DK\mid Y\preceq Z\}.
\]
\end{definition}

The domination relation is closed under Box convergence \cite[Theorem~4.16]{gds1}, and hence $\mathcal P_Z$ is Box closed. Reflexivity and transitivity of the feature order show that it is downward closed, and $Z$ is a common upper bound for any two of its members. Thus $\mathcal P_Z\in\Pi_K$.

For a closed subset $\mathcal A\subset\DK$ and $Z\in\DK$, we write
\[
\Box(Z,\mathcal A)\coloneqq\inf_{Y\in\mathcal A}\Box(Z,Y).
\]

\begin{definition}[Weak convergence of pyramids {\cite[Definition~2.11]{gds1}}]
\label{def:weak-convergence}
Let $\mathcal P_n,\mathcal P\subset\DK$ be closed. We say that $\mathcal P_n$ \emph{converges weakly} to $\mathcal P$ as $n\to\infty$ if
\[
\Box(Z,\mathcal P_n)\longrightarrow0\qquad(n\to\infty)
\]
for every $Z\in\mathcal P$, and
\[
\liminf_{n\to\infty}\Box(Z,\mathcal P_n)>0
\]
for every $Z\in\DK\setminus\mathcal P$.
\end{definition}

This is sequential Painlev\'e--Kuratowski convergence, also called weak Hausdorff convergence in mm-space theory.

\begin{definition}[Finite-measurement set]
\label{def:finite-measurements}
For $\mathcal P\in\Pi_K$ and $N\in\mathbb N$, the \emph{finite-measurement set} at level $N$ is
\[
\DM_K(\mathcal P;N)\coloneqq\mathcal P\cap\DM_K(N).
\]
For a general $X\in\D$, put
\[
\DM(X;N)\coloneqq\{Y\in\D\mid Y\preceq X,\ \#F_Y\leq N\}.
\]
\end{definition}

The set $\DM_K(\mathcal P;N)$ is nonempty and Box compact. A one-feature quotient of any member of $\mathcal P$ proves nonemptiness, while \Cref{lem:bounded-measurement-compact} and closedness of $\mathcal P$ prove compactness.

If $X\in\DK$, then every geometric data set dominated by $X$ also belongs to $\DK$. Indeed, the image of a domination map is dense by full support. Pullback sends every feature of the dominated object to a $K$-valued feature of $X$, so continuity and closedness of $K$ show that the original feature is $K$-valued. Therefore,
\[
\DM(X;N)=\DM_K(\mathcal P_X;N).
\]

\subsection{Three stabilizations of weak limits}
\label{subsec:pyramid-weak-limits}

To retain downward closedness and directedness under weak limits, domination must be transferred to approximating sequences and common upper bounds must be replaced by a Box-precompact family. The next three lemmas perform these operations without leaving the fixed range $K$.

\begin{lemma}
\label{lem:K-domination-refinement}
Let $\bar X,Y,\bar Y\in\DK$ and suppose that $Y\preceq\bar Y$. Then there is $X\in\DK$ such that
\[
X\preceq\bar X,\qquad\Box(X,Y)\leq\Box(\bar X,\bar Y).
\]
If $Y$ has at most $N$ features, then $X$ may also be chosen with at most $N$ features.
\end{lemma}

The refinement with the Box estimate is \cite[Lemma~5.4]{gds2}. We verify in \Cref{subsec:app-K-refinement} that its quotient construction preserves the range $K$ and the number of features.

\begin{lemma}
\label{lem:finite-net-precompactness}
Let $\mathcal E\subset\DK$. Suppose that, for every $\epsilon>0$, there is $N\in\mathbb N$ such that, for each $Z\in\mathcal E$, one can choose a closed set $L_Z\subset Z$ and a finite set $G_Z\subset\overline{F_Z}$ satisfying
\[
\mu_Z(L_Z)>1-\epsilon,\qquad\#G_Z\leq N,\qquad(d_{\infty,L_Z})_H(\overline{F_Z},G_Z)<\epsilon.
\]
Then $\mathcal E$ is Box precompact.
\end{lemma}

This is the finite-net criterion in which each feature is measured by itself \cite[Lemma~4.12]{gds2}. The correspondence between its assumptions and the three conditions above is recorded in \Cref{subsec:app-K-refinement}.

\begin{lemma}
\label{lem:K-common-upper-bound}
Suppose that $X_n,Y_n,\bar Z_n,X,Y\in\DK$ satisfy
\[
X_n\preceq\bar Z_n,\qquad Y_n\preceq\bar Z_n,\qquad\Box(X_n,X)+\Box(Y_n,Y)\longrightarrow0\qquad(n\to\infty).
\]
Then one can choose $Z_n\in\DK$ such that
\[
X_n,Y_n\preceq Z_n\preceq\bar Z_n
\]
and $(Z_n)$ has a Box-convergent subsequence.
\end{lemma}

The precompact refinement of common upper bounds is \cite[Lemma~5.6]{gds2}. We check in \Cref{subsec:app-K-refinement} that the quotient constructed there remains in $\DK$.

\begin{theorem}
\label{thm:K-pyramid-weak-limit}
Every weak limit of a sequence of $K$-pyramids is a $K$-pyramid.
\end{theorem}

\begin{proof}
Suppose that $\mathcal P_n\in\Pi_K$ converges weakly to a closed subset $\mathcal P\subset\DK$ as $n\to\infty$. Take $Y\in\mathcal P$ and $X\preceq Y$. Choose $Y_n\in\mathcal P_n$ such that $Y_n\to Y$ in the Box distance as $n\to\infty$. Applying \Cref{lem:K-domination-refinement} with $\bar X=Y_n$, $Y=X$, and $\bar Y=Y$ gives $X_n\in\DK$ such that $X_n\preceq Y_n$ and
\[
\Box(X_n,X)\leq\Box(Y_n,Y).
\]
Downward closedness of each $\mathcal P_n$ gives $X_n\in\mathcal P_n$, and weak convergence gives $X\in\mathcal P$. Thus $\mathcal P$ is downward closed.

Take $X,Y\in\mathcal P$, and choose $X_n,Y_n\in\mathcal P_n$ converging in the Box distance to $X,Y$, respectively, as $n\to\infty$. Directedness of $\mathcal P_n$ gives $\bar Z_n\in\mathcal P_n$ such that $X_n,Y_n\preceq\bar Z_n$. By \Cref{lem:K-common-upper-bound}, there are $Z_n\in\DK$ such that $X_n,Y_n\preceq Z_n\preceq\bar Z_n$ and, after passage to a subsequence, $Z_n\to Z$ in the Box distance as $n\to\infty$. Downward closedness gives $Z_n\in\mathcal P_n$, and weak convergence gives $Z\in\mathcal P$. Closedness of domination \cite[Theorem~4.16]{gds1} gives $X,Y\preceq Z$. Therefore, $\mathcal P$ is directed.

For each $n$, choose $W_n\in\mathcal P_n$ and $f_n\in F_{W_n}$. The one-feature quotient
\[
V_n\coloneqq W_n/\{f_n\}
\]
belongs to $\mathcal P_n\cap\DM_K(1)$. By \Cref{lem:bounded-measurement-compact}, after passage to a subsequence, $V_n\to V$ in the Box distance for some $V\in\DM_K(1)$ as $n\to\infty$. Weak convergence gives $V\in\mathcal P$, and hence $\mathcal P$ is nonempty. Closedness is part of the assumption on the limit set, so $\mathcal P\in\Pi_K$. This completes the proof.
\end{proof}

\subsection{Detection and reconstruction by finite measurements}
\label{subsec:pyramid-reconstruction}

Weak convergence can be detected on each finite-feature layer. For principal pyramids, convergence of every finite-measurement layer also reconstructs convergence of the generators in the observable distance.

\begin{lemma}
\label{lem:K-pyramid-measurements}
For $K$-pyramids $\mathcal P_n,\mathcal P$, the following conditions are equivalent.
\begin{enumerate}
\item The sequence $\mathcal P_n$ converges weakly to $\mathcal P$ as $n\to\infty$.
\item For every $N\in\mathbb N$,
\[
(\Box)_H(\DM_K(\mathcal P_n;N),\DM_K(\mathcal P;N))\longrightarrow0\qquad(n\to\infty).
\]
\end{enumerate}
\end{lemma}

\begin{proof}
Assume weak convergence. For $A\in\DM_K(\mathcal P;N)$, choose $Y_n\in\mathcal P_n$ such that $Y_n\to A$ in the Box distance as $n\to\infty$. Apply \Cref{lem:K-domination-refinement} with $\bar X=Y_n$ and $Y=\bar Y=A$. This gives $A_n\in\DM_K(\mathcal P_n;N)$ such that $\Box(A_n,A)\leq\Box(Y_n,A)$. The reverse Hausdorff approximation follows by taking a subsequential limit of a counterexample sequence in the compact set $\DM_K(N)$ and using the second condition of weak convergence. The compactness argument that makes both approximations uniform is given in \Cref{subsec:app-pyramid-measurements}.

Conversely, suppose that the finite-measurement sets converge. Approximate $X\in\mathcal P$ by a finite-feature quotient $A\preceq X$ using \Cref{subsec:app-screened-closed-density}, and then approximate $A$ by $A_n\in\mathcal P_n$ in the same finite layer. This gives $\Box(X,\mathcal P_n)\to0$ as $n\to\infty$. Now suppose that $Y_n\in\mathcal P_n$ converges to $Y\in\DK$ in the Box distance as $n\to\infty$. Choose finite quotients $B_m\preceq Y$ converging to $Y$ in the Box distance as $m\to\infty$, and use \Cref{lem:K-domination-refinement} to transfer approximations of $B_m$ below $Y_n$. For each fixed $m$, convergence of the finite-measurement sets gives $B_m\in\mathcal P$. Letting $m\to\infty$ and using closedness of $\mathcal P$ gives $Y\in\mathcal P$. This is the second condition of weak convergence. This completes the proof.
\end{proof}

\begin{lemma}
\label{lem:finite-measurement-proximity}
Let $X,Y\in\DK$ and $N\in\mathbb N$. For every $A\in\DM(X;N)$, there is $B\in\DM(Y;N)$ such that
\[
\Box(A,B)\leq2N\dconc(A,B)\leq2N\dconc(X,Y).
\]
In particular,
\[
(\Box)_H(\DM(X;N),\DM(Y;N))\leq2N\dconc(X,Y).
\]
\end{lemma}

The first quotient selection is \cite[Lemma~6.10]{gds2}, and the Box estimate between finite-feature objects is \cite[Proposition~6.8]{gds2}. We verify the fixed-range specialization in \Cref{subsec:app-pyramid-reconstruction}.

\begin{lemma}
\label{lem:finite-measurement-reconstruction}
Let $X,X_n\in\DK$. If
\[
(\Box)_H(\DM(X_n;N),\DM(X;N))\longrightarrow0\qquad(n\to\infty)
\]
for every $N\in\mathbb N$, then $\dconc(X_n,X)\to0$ as $n\to\infty$.
\end{lemma}

This is the finite-measurement reconstruction theorem for the family consisting only of the identity transformation \cite[Proposition~6.24]{gds2}. The correspondence with the notation used here is given in \Cref{subsec:app-pyramid-reconstruction}.

\subsection{Compactification by principal pyramids}
\label{subsec:pyramid-main-compactification}

For $\mathcal P,\mathcal Q\in\Pi_K$, define
\begin{equation}
\label{eq:pyramid-metric}
d_{\Pi,K}(\mathcal P,\mathcal Q)\coloneqq\sum_{N=1}^\infty\frac{1}{2N2^N}(\Box)_H(\DM_K(\mathcal P;N),\DM_K(\mathcal Q;N)).
\end{equation}
Every measurement set is nonempty and compact. The series converges because $\Box\leq2$. For $Z\in\DK$, put
\[
\iota_K(Z)\coloneqq\mathcal P_Z.
\]

\begin{theorem}
\label{thm:K-pyramid-compactification}
The formula in \Cref{eq:pyramid-metric} defines a metric on $\Pi_K$, and $(\Pi_K,d_{\Pi,K})$ is compact. The map
\[
\iota_K\colon(\DK,\dconc)\longrightarrow(\Pi_K,d_{\Pi,K})
\]
is a $1$-Lipschitz topological embedding with dense image. Moreover, the composite map
\[
\iota_K\circ\mathsf C_\tau\colon(\D,\dtau)\longrightarrow(\Pi_K,d_{\Pi,K})
\]
is a topological embedding with dense image.
\end{theorem}

\begin{proof}
We divide the proof into four steps.

First, \Cref{lem:K-pyramid-measurements} shows that Hausdorff convergence in every finite-measurement coordinate is equivalent to weak convergence of pyramids. The bound $\Box\leq2$ and the tail estimate for the series show that this is also equivalent to convergence in $d_{\Pi,K}$. Details of the series estimate are given in \Cref{subsec:app-pyramid-reconstruction}. If $d_{\Pi,K}(\mathcal P,\mathcal Q)=0$, then the constant sequence $\mathcal P_n=\mathcal P$ converges weakly to $\mathcal Q$. The two conditions in \Cref{def:weak-convergence} give $\mathcal P=\mathcal Q$. The other metric axioms follow from those of the Hausdorff distances.

Second, take a sequence $(\mathcal P_n)\subset\Pi_K$. Since $(\DK,\Box)$ is Polish, the subsequence theorem for sequential Painlev\'e--Kuratowski convergence of closed sets \cite[Theorem~5.2.12]{beer1993topologies} gives a subsequence converging weakly to a closed subset $\mathcal P\subset\DK$. By \Cref{thm:K-pyramid-weak-limit}, the limit satisfies $\mathcal P\in\Pi_K$. The same subsequence converges in $d_{\Pi,K}$ by \Cref{lem:K-pyramid-measurements}. Therefore, $\Pi_K$ is compact.

Third, take $X,Y\in\DK$. The identity $\DM(X;N)=\DM_K(\mathcal P_X;N)$ and \Cref{lem:finite-measurement-proximity}, applied in both directions to $X$ and $Y$, give
\[
(\Box)_H(\DM_K(\mathcal P_X;N),\DM_K(\mathcal P_Y;N))\leq2N\dconc(X,Y).
\]
Multiplying by the coefficient in \Cref{eq:pyramid-metric} and summing over $N$ shows that $\iota_K$ is $1$-Lipschitz. Conversely, suppose that $d_{\Pi,K}(\mathcal P_{X_n},\mathcal P_X)\to0$ as $n\to\infty$. Then every finite-measurement layer converges in the Hausdorff Box distance. Applying \Cref{lem:finite-measurement-reconstruction} to $X_n$ and $X$ gives $\dconc(X_n,X)\to0$ as $n\to\infty$. Thus $\iota_K$ is a topological embedding.

Fourth, fix $\mathcal P\in\Pi_K$. Separability of $\DK$ gives a dense sequence $(Y_n)$ in $\mathcal P$. Set $Z_1\coloneqq Y_1$. Using directedness, choose $Z_n\in\mathcal P$ inductively for $n\geq2$ so that $Y_n\preceq Z_n$ and $Z_{n-1}\preceq Z_n$. Given $W\in\mathcal P$ and $\epsilon>0$, choose $m$ such that $\Box(W,Y_m)<\epsilon$. Then $Y_m\in\mathcal P_{Z_n}$ for every $n\geq m$. Conversely, if $W_n\in\mathcal P_{Z_n}$ converges in the Box distance along a subsequence, then $W_n\in\mathcal P$ and closedness places its limit in $\mathcal P$. Thus $\mathcal P_{Z_n}$ converges weakly to $\mathcal P$, and
\[
d_{\Pi,K}(\mathcal P_{Z_n},\mathcal P)\longrightarrow0\qquad(n\to\infty).
\]
Therefore, $\iota_K(\DK)$ is dense. Finally, \Cref{lem:screened-topology} shows that $\mathsf C_\tau\colon(\D,\dtau)\to(\DK,\dconc)$ is a topological embedding. Its image is Box dense by \Cref{thm:screen-density}, and hence it is also dense for the observable distance because $\dconc\leq\Box$. The composite $\iota_K\circ\mathsf C_\tau$ is therefore a topological embedding with dense image. This completes the proof.
\end{proof}

\begin{corollary}
\label{cor:tau-pyramid-comparison}
For all $X,Y\in\D$,
\[
d_{\Pi,K}\bigl(\iota_K(\mathsf C_\tau(X)),\iota_K(\mathsf C_\tau(Y))\bigr)\leq\min\{\Box_\tau(X,Y),\sqrt{\dtau(X,Y)}\}.
\]
\end{corollary}

\begin{proof}
The $1$-Lipschitz estimate in \Cref{thm:K-pyramid-compactification} bounds the left-hand side by $\dconc^\tau(X,Y)$. Applying $\dconc\leq\Box$ to the two compact screens gives $\dconc^\tau(X,Y)\leq\Box_\tau(X,Y)$. By \Cref{lem:screened-topology}, we also have $\dconc^\tau(X,Y)\leq\sqrt{\dtau(X,Y)}$. This completes the proof.
\end{proof}

\section{Tensorization over a Common Factor}
\label{sec:tensorization}

The bounded coordinate is stable not only under finite-observation compactification but also under products with a common factor. Independent approximation of product features on different slices can destroy the Lipschitz condition between slices. We repair that condition by a pooling operation that does not increase either relevant error.

\subsection{The sum-metric product and incompatible slices}
\label{subsec:l1-slice-obstacle}

\begin{definition}[$\ell_1$ product {\cite[Definition~3.66]{shioya2024sgc}}]
\label{def:l1-product}
For mm-spaces $X$ and $W$, their \emph{$\ell_1$ product} is
\[
X\times_1W\coloneqq(X\times W,d_{X\times_1W},\mu_X\otimes\mu_W),
\]
where
\[
d_{X\times_1W}((x,w),(x',w'))\coloneqq d_X(x,x')+d_W(w,w').
\]
We regard this mm-space as the geometric data set whose features are all real-valued $1$-Lipschitz functions.
\end{definition}

Let $W=\{w_1,\ldots,w_N\}$ be a finite mm-space, and write
\[
\mu_W=\sum_{i=1}^Na_i\delta_{w_i},\qquad a_i>0.
\]
For $F\in\operatorname{Lip}_1(X\times_1W)$, put $f_i(x)\coloneqq F(x,w_i)$. Each slice $f_i$ is $1$-Lipschitz on $X$, and
\[
|f_i(x)-f_j(x)|\leq d_W(w_i,w_j)
\]
for every $x\in X$ and all $i,j$. Conversely, a family of slices satisfying these two conditions defines a $1$-Lipschitz function on $X\times_1W$.

Independent approximation of the slices need not preserve the second condition. Consider a two-point factor with $d_W(w_1,w_2)=r>0$, and suppose that $f_1=r/2$ and $f_2=-r/2$ at one point. The independent approximations $r/2+\eta$ and $-r/2-\eta$ each have error $\eta$, while their difference is $r+2\eta$. Thus arbitrarily small independent errors violate the factor-direction Lipschitz condition at the boundary of the constraint.

\subsection{Pooling restores the constraint}
\label{subsec:pooling}

The feasible region for the slice values and its weighted bounded-coordinate metric are
\[
K_W\coloneqq\{z\in\R^N\mid |z_i-z_j|\leq d_W(w_i,w_j)\text{ for all }i,j\},
\]
and
\[
d_{\tau,a}(z,z')\coloneqq\sum_{i=1}^Na_i|\tau(z_i)-\tau(z_i')|.
\]
We call the repair that moves $\R^N$ into $K_W$ \emph{pooling}. It will preserve the slice constraints without increasing either $d_{\tau,a}$ or the uniform distance.

\begin{lemma}
\label{lem:order-mass-contraction}
Let $J$ be an interval, let $a_1,\ldots,a_N$ be positive, and let $S\colon J^N\to J^N$ be coordinatewise nondecreasing. Suppose that
\[
\sum_{k=1}^Na_k(Su)_k=\sum_{k=1}^Na_ku_k
\]
for every $u\in J^N$. Then
\[
\sum_{k=1}^Na_k|(Su)_k-(Sv)_k|\leq\sum_{k=1}^Na_k|u_k-v_k|
\]
for all $u,v\in J^N$.
\end{lemma}

This is \cite[Proposition~1, p.~385]{crandall-tartar1980nonexpansiveorder} applied to the finite measure space in which the atom $k$ has mass $a_k$ and to the subset $J^N\subset L^1$.

Fix a strictly increasing homeomorphism $\sigma\colon\R\to J$ onto an open interval and indices $i<j$. Put $r_{ij}\coloneqq d_W(w_i,w_j)$. The map $P_{ij}^\sigma\colon\R^N\to\R^N$ leaves all coordinates other than $i,j$ unchanged and fixes $z$ when $|z_i-z_j|\leq r_{ij}$. If $z_i>z_j+r_{ij}$, take the unique $t\in[z_j,z_i-r_{ij}]$ satisfying
\[
a_i\sigma(t+r_{ij})+a_j\sigma(t)=a_i\sigma(z_i)+a_j\sigma(z_j),
\]
and replace $(z_i,z_j)$ by $(t+r_{ij},t)$. Strict monotonicity and the two endpoint inequalities give existence and uniqueness of $t$. For a violation in the opposite direction, interchange $i$ and $j$.

\begin{lemma}
\label{lem:pairwise-pooling}
The map $P_{ij}^\sigma$ fixes every vector satisfying the $(i,j)$ constraint, and
\[
\sum_{k=1}^Na_k|\sigma((P_{ij}^\sigma z)_k)-\sigma((P_{ij}^\sigma z')_k)|\leq\sum_{k=1}^Na_k|\sigma(z_k)-\sigma(z'_k)|
\]
for all $z,z'\in\R^N$.
\end{lemma}

\begin{proof}
The first assertion follows from the definition. The three regions and their boundaries are examined in \Cref{subsec:app-pairwise-pooling}. The calculation shows that each output coordinate of $P_{ij}^\sigma$ is nondecreasing in every input coordinate. Therefore,
\[
S\coloneqq\sigma^N\circ P_{ij}^\sigma\circ(\sigma^{-1})^N\colon J^N\longrightarrow J^N
\]
is coordinatewise nondecreasing. The defining equation for $P_{ij}^\sigma$ shows that $S$ preserves the weighted mass. Applying \Cref{lem:order-mass-contraction} to this map gives the asserted inequality. This completes the proof.
\end{proof}

From now on, set $\sigma=\tau$ and write $P_{ij}\coloneqq P_{ij}^\tau$. For $s\geq0$, define $T_s\colon(-1/2,1/2)\to(-1/2,1/2)$ by
\[
T_s(u)\coloneqq\tau(\tau^{-1}(u)+s).
\]

\begin{lemma}
\label{lem:tau-concavity}
The map $T_s$ is increasing and concave. Moreover,
\[
\|P_{ij}z-P_{ij}z'\|_\infty\leq\|z-z'\|_\infty
\]
for all $z,z'\in\R^N$.
\end{lemma}

\begin{proof}
The addition formula for the hyperbolic tangent gives
\[
\begin{aligned}
T_s(u)&=\frac{u+(\tanh s)/2}{1+2u\tanh s},\\
T_s'(u)&=\frac{1-\tanh^2s}{(1+2u\tanh s)^2},\\
T_s''(u)&=-\frac{4\tanh s(1-\tanh^2s)}{(1+2u\tanh s)^3}.
\end{aligned}
\]
For $-1/2<u<1/2$ and $s\geq0$, the first derivative is positive and the second derivative is nonpositive.

Pooling preserves the weighted mean of the two $\tau$ coordinates while moving them closer. Applying concavity of $T_s$ to these two points gives
\[
P_{ij}(z+s\boldsymbol1)\leq P_{ij}z+s\boldsymbol1,
\]
where $\boldsymbol1\coloneqq(1,\ldots,1)$. The weighted two-point comparison and the two violation directions are verified in \Cref{subsec:app-pairwise-pooling}. Since $P_{ij}$ is coordinatewise nondecreasing, $\|z-z'\|_\infty\leq s$ implies
\[
P_{ij}z'\leq P_{ij}(z+s\boldsymbol1)\leq P_{ij}z+s\boldsymbol1.
\]
Interchanging $z$ and $z'$ proves the uniform estimate. This completes the proof.
\end{proof}

\begin{theorem}
\label{thm:pooling-retraction}
There is a retraction $P_W\colon\R^N\to K_W$ such that
\[
d_{\tau,a}(P_Wz,P_Wz')\leq d_{\tau,a}(z,z'),\qquad\|P_Wz-P_Wz'\|_\infty\leq\|z-z'\|_\infty
\]
for all $z,z'\in\R^N$.
\end{theorem}

\begin{proof}
For $N=1$, take the identity map. Suppose that $N\geq2$, enumerate the unordered pairs as $e_1,\ldots,e_m$, and write $P_e$ for the pooling map associated with a pair $e$. Starting from $z^{(0)}\coloneqq z$, define
\[
z^{(n+1)}\coloneqq P_{e_{1+(n\bmod m)}}z^{(n)}.
\]
Each step replaces the two affected coordinates by values between them. Hence every iterate lies in $[\min_i z_i,\max_i z_i]^N$.

The energy
\[
E(q)\coloneqq\sum_{i=1}^Na_i\tau(q_i)^2
\]
decreases at every step, and its decrease vanishes exactly when the current pair satisfies its constraint. In \Cref{subsec:app-pooling-convergence}, continuity of the decrease is used to prove $\|z^{(n+1)}-z^{(n)}\|_\infty\to0$ as $n\to\infty$ and then to show that every cluster point satisfies all pairwise constraints. Nonexpansiveness from \Cref{lem:pairwise-pooling} then forces the full sequence to converge to the same cluster point. Denote this limit by $P_Wz$.

If $z\in K_W$, then every step fixes $z$, so $P_W$ is a retraction. Apply the same pooling sequence to two initial vectors and use \Cref{lem:pairwise-pooling,lem:tau-concavity} at every step. Letting $n\to\infty$ gives the two displayed inequalities. This completes the proof.
\end{proof}

\subsection{Finite and general factors}
\label{subsec:factor-contraction}

Let $X,Y$ be mm-spaces and $\pi\in\T(\mu_X,\mu_Y)$. For $f\in\operatorname{Lip}_1(X)$ and $g\in\operatorname{Lip}_1(Y)$, put
\[
c_\tau^\pi(f,g)\coloneqq\int_{X\times Y}|\tau(f(x))-\tau(g(y))|\,d\pi(x,y),
\]
and define the directed fixed-coupling quantity by
\[
\overrightarrow D_\tau^\pi(X,Y)\coloneqq\sup_{f\in\operatorname{Lip}_1(X)}\inf_{g\in\operatorname{Lip}_1(Y)}c_\tau^\pi(f,g).
\]
If $\pi^{\mathrm{op}}\coloneqq(\pr_2,\pr_1)_*\pi$, then the fixed-coupling quantity in \Cref{def:dtau} is
\[
D_\tau^\pi(X,Y)=\max\{\overrightarrow D_\tau^\pi(X,Y),\overrightarrow D_\tau^{\pi^{\mathrm{op}}}(Y,X)\}.
\]
We also write
\[
\pi_W^\Delta\coloneqq(\operatorname{id}_W,\operatorname{id}_W)_*\mu_W
\]
for the diagonal coupling of $W$.

\begin{lemma}
\label{lem:finite-factor-contraction}
Let $X,Y$ be mm-spaces, let $W$ be a finite mm-space, and let $\pi\in\T(\mu_X,\mu_Y)$. Then
\[
\overrightarrow D_\tau^{\pi\otimes\pi_W^\Delta}(X\times_1W,Y\times_1W)\leq\overrightarrow D_\tau^\pi(X,Y).
\]
\end{lemma}

\begin{proof}
Put $\delta\coloneqq\overrightarrow D_\tau^\pi(X,Y)$. Given $F\in\operatorname{Lip}_1(X\times_1W)$, put $f_i(x)\coloneqq F(x,w_i)$. Then $f(x)\coloneqq(f_i(x))_{i=1}^N$ belongs to $K_W$. For $\epsilon>0$, choose $h_i\in\operatorname{Lip}_1(Y)$ such that
\[
c_\tau^\pi(f_i,h_i)\leq\delta+\epsilon.
\]

Put $h(y)\coloneqq(h_1(y),\ldots,h_N(y))$ and $g(y)\coloneqq P_Wh(y)$. Uniform nonexpansiveness in \Cref{thm:pooling-retraction} shows that each component $g_i$ is $1$-Lipschitz in the $Y$ direction. Define $G(y,w_i)\coloneqq g_i(y)$. Since $g(y)\in K_W$, we have
\[
\lvert G(y,w_i)-G(y',w_j)\rvert
\leq d_Y(y,y')+d_W(w_i,w_j).
\]
Thus $G$ is $1$-Lipschitz on $Y\times_1W$. The identities $P_Wf(x)=f(x)$ and the weighted-distance estimate give
\[
c_\tau^{\pi\otimes\pi_W^\Delta}(F,G)\leq\sum_{i=1}^Na_i c_\tau^\pi(f_i,h_i)\leq\delta+\epsilon.
\]
Taking the infimum over $G$, the supremum over $F$, and then letting $\epsilon\downarrow0$ proves the assertion. This completes the proof.
\end{proof}

\begin{theorem}
\label{thm:tau-tensorization}
For mm-spaces $X,Y,W$ and $\pi\in\T(\mu_X,\mu_Y)$,
\[
D_\tau^{\pi\otimes\pi_W^\Delta}(X\times_1W,Y\times_1W)=D_\tau^\pi(X,Y).
\]
Moreover,
\[
\dtau(X\times_1W,Y\times_1W)\leq\dtau(X,Y).
\]
\end{theorem}

\begin{proof}
We first prove contraction. In \Cref{subsec:app-factor-approximation}, the measure of $W$ is approximated by finitely supported probability measures, and the directed fixed-coupling quantity is shown to be continuous under this approximation. Applying \Cref{lem:finite-factor-contraction} to every finite support in both directions and passing to the limit gives
\[
D_\tau^{\pi\otimes\pi_W^\Delta}(X\times_1W,Y\times_1W)\leq D_\tau^\pi(X,Y).
\]

For the reverse directed estimate, lift $f\in\operatorname{Lip}_1(X)$ to the product by $F(x,w)\coloneqq f(x)$. Given $G\in\operatorname{Lip}_1(Y\times_1W)$, put
\[
u(y)\coloneqq\int_W\tau(G(y,w))\,d\mu_W(w),\qquad g(y)\coloneqq\tau^{-1}(u(y)).
\]
The integrand takes values in $(-1/2,1/2)$. If $u(y)$ were an endpoint, then the integrand would equal that endpoint almost everywhere, contrary to its range. Thus $u(y)\in(-1/2,1/2)$. Let $r\coloneqq d_Y(y,y')$. The inequalities $G(y,w)\leq G(y',w)+r$ and \Cref{lem:tau-concavity} give
\[
\tau(G(y,w))\leq T_r(\tau(G(y',w))).
\]
Applying Jensen's inequality for the concave function $T_r$ \cite[Section~1.2.1, Equation~(1.2.1), p.~10]{bakry-gentil-ledoux2014markov} gives
\[
u(y)\leq T_r(u(y'))=\tau(g(y')+r).
\]
Since $\tau$ is increasing, $g(y)\leq g(y')+r$. Interchanging $y$ and $y'$ shows that $g\in\operatorname{Lip}_1(Y)$.

For every $(x,y)\in X\times Y$,
\[
|\tau(f(x))-\tau(g(y))|\leq\int_W|\tau(F(x,w))-\tau(G(y,w))|\,d\mu_W(w).
\]
Integrating with respect to $\pi$, taking the infimum over $G$, and taking the supremum over $f$ gives the reverse directed estimate. Interchanging $X$ and $Y$ gives equality for the fixed coupling. Finally, $\pi\otimes\pi_W^\Delta$ is a coupling of the products. Apply the contraction estimate to every $\pi\in\T(\mu_X,\mu_Y)$ and take the infimum over $\pi$ to obtain the inequality for $\dtau$. This completes the proof.
\end{proof}

For the prescribed coupling, adjoining the diagonal coupling of the common factor preserves the fixed-coupling quantity exactly. The distance $\dtau$ takes the infimum over all couplings of the products, so only the nonexpansive inequality remains. The proof uses the sum metric and all real-valued $1$-Lipschitz functions on mm-spaces. General geometric data sets or other product metrics would require additional closure properties.

\appendix

\section{Bounded Features, Inversion, and the Screened Class}
\label{app:bounded-feature-inversion}

This appendix collects the compactness, realization, optimal-coupling, feature-closure, and inversion arguments used in \Cref{sec:bounded-observable-distance}, together with the fixed-range closedness and finite-feature approximation used in \Cref{sec:compact-screen-box}.

\subsection{Bounded Lipschitz families and feature realization}
\label{subsec:app-bounded-realization}

\begin{proof}[Proof of \Cref{lem:bounded-lipschitz-compactness}]
Put $c\coloneqq\max\{1,\ell\}$ and replace the metric of $Z$ by $cd$. Take an integer $k\geq B$. Then $\mathcal A$ is contained in the family of $1$-Lipschitz functions bounded in absolute value by $k$. By \cite[Lemma~2.2 and Remark~2.3]{hanika2022gds}, pointwise convergence and Ky Fan convergence agree on this family, and the pointwise closure $\mathcal C$ of $\mathcal A$ is compact. Every pointwise limit is continuous, $\ell$-Lipschitz, and bounded in absolute value by $B$.

Suppose that $u_n\to u$ pointwise in $\mathcal C$ as $n\to\infty$. Applying the dominated convergence theorem \cite[Theorem~3.31, pp.~92--93]{axler2020mira} to $|u_n-u|\leq2B$ shows that $u_n\to u$ in $L^1(\mu)$. Since $\mu$ has full support, two continuous functions that agree almost everywhere agree everywhere. The natural map from $\mathcal C$ to $L^1(\mu)$ is therefore a continuous injection from a compact space into a Hausdorff space. Its image is the $L^1(\mu)$-closure of $\mathcal A$, which also proves the identification with the pointwise closure. This completes the proof.
\end{proof}

\begin{proof}[Proof of \Cref{lem:feature-realization}]
The quotient construction for a geometric data set is given in \cite[Definition~3.16, Proposition~3.17, Claim~3.18, and Proposition~3.19]{gds1}. We verify the extension needed when the domain $Z$ is only assumed to be separable and metrizable.

Every $h\in H$ is constant on the zero-distance classes of $d^H$ and is $1$-Lipschitz on $Z_H$. It therefore has a unique continuous extension $\widehat h$ to $\widehat Z_H$. Fix $z_0\in Z$ and consider
\[
Z_H\longrightarrow\ell^\infty(H),\qquad
[z]\longmapsto(h(z)-h(z_0))_{h\in H}.
\]
Its image is bounded because $d^H(z,z_0)<\infty$, and the supremum-norm distance between the images of $[z]$ and $[z']$ equals $d^H(z,z')$. This is an isometric embedding and extends isometrically to $\widehat Z_H$. Its coordinates are $\widehat h-\widehat h(q(z_0))$, so the extended feature family induces the metric on the completion.

The space $\widehat Z_H$ is complete and separable, and $Z_H^\nu$ is a closed subspace. The assumed Borel measurability of the canonical map makes $q_*\nu$ well-defined, and its restriction to its support has full support. The stated triple is therefore a geometric data set. If $J$ is closed, then every $\widehat h$ also takes values in $J$. Finally, suppose that $\nu$ has full support and $q$ is continuous. The inverse image under $q$ of every neighborhood of $q(z)$ is a positive-measure neighborhood of $z$. Therefore, $q(Z)\subset Z_H^\nu$. This completes the proof.
\end{proof}

\subsection{Compactified features and optimal couplings}
\label{subsec:app-feature-coupling}

We first prove attainment of the infimum used in \Cref{thm:dtau-metric}. By \Cref{lem:feature-compactness}, the sets $\Kc_\tau(X)$ and $\Kc_\tau(Y)$ are compact in their respective $L^1$ spaces. Fix $\epsilon>0$, and let $\{u_1,\ldots,u_m\}$ and $\{v_1,\ldots,v_l\}$ be finite $\epsilon$-nets. For every coupling $\pi$ and features $u,u',v,v'$,
\[
\begin{aligned}
&\left|\int|u\circ\pr_1-v\circ\pr_2|\,d\pi
-\int|u'\circ\pr_1-v'\circ\pr_2|\,d\pi\right|\\
&\qquad\leq
\|u-u'\|_{L^1(\mu_X)}+\|v-v'\|_{L^1(\mu_Y)}.
\end{aligned}
\]

Suppose that $\pi_k$ converges weakly to $\pi$ as $k\to\infty$. For every pair of net centers, the function $(x,y)\mapsto|u_i(x)-v_j(y)|$ is bounded and continuous, so its integrals converge. Let $\eta_k$ be the maximum of the absolute differences of these integrals over the finitely many pairs. Then $\eta_k\to0$ as $k\to\infty$. Applying the net approximation on both sides gives
\[
\sup_{u,v}\left|
\int|u\circ\pr_1-v\circ\pr_2|\,d\pi_k
-\int|u\circ\pr_1-v\circ\pr_2|\,d\pi
\right|\leq4\epsilon+\eta_k.
\]
Since $\epsilon>0$ is arbitrary, the family of costs converges uniformly. Both directed Hausdorff functionals therefore converge, and $\pi\mapsto D_\tau^\pi(X,Y)$ is continuous. Weak compactness of the coupling set \cite[Lemma~2.10]{gds1} proves attainment.

We next identify the feature closure of the intermediate object in \Cref{thm:dtau-geodesic}. The map
\[
\Phi_t(u,v)=(1-t)u+tv
\]
is continuous on $R$. Thus $C_t=\Phi_t(R)$ is compact, and $\Phi_t(R_0)$ is dense in $C_t$. Transforming the generating features of $H_t$ by $\tau$ and pulling them back to $Z$ gives exactly $\Phi_t(R_0)$. The canonical map $q_t$ pushes $\pi$ forward to $\mu_t$. By \Cref{lem:feature-realization} and the identity $Q_t=\supp\mu_t$, the set $q_t(Z)$ is contained in and dense in $Q_t$. The transformed features are uniformly bounded and satisfy a common Lipschitz estimate. Therefore, \Cref{lem:bounded-lipschitz-compactness} identifies their pointwise closure with their $L^1(Z,\pi)$ closure, which is $C_t$.

\subsection{Interpolation and inversion}
\label{subsec:app-interpolation-inversion}

\begin{proof}[Proof of \Cref{lem:scalar-interpolation}]
Put $x\coloneqq\tanh a$, $y\coloneqq\tanh b$, and $q\coloneqq(1-t)x+ty$. Then
\[
M_t(a,b)=\frac12\log\frac{1+q}{1-q},
\]
and
\[
\frac{\partial M_t}{\partial a}
=\frac{(1-t)(1-x^2)}{1-q^2},\qquad
\frac{\partial M_t}{\partial b}
=\frac{t(1-y^2)}{1-q^2}.
\]
The inequalities
\[
\begin{aligned}
1-q&\geq(1-t)(1-x),&\qquad 1+q&\geq(1-t)(1+x),\\
1-q&\geq t(1-y),&1+q&\geq t(1+y)
\end{aligned}
\]
give
\[
1-q^2\geq(1-t)^2(1-x^2),\qquad
1-q^2\geq t^2(1-y^2).
\]
The two partial derivatives are therefore bounded above by $(1-t)^{-1}$ and $t^{-1}$, respectively. Varying the two variables in succession gives the two-variable estimate. For an endpoint constant, the same calculation applies with $x=\pm1$ or $y=\pm1$. This completes the proof.
\end{proof}

\begin{proof}[Proof of \Cref{lem:inversion-tightness}]
Fix $\epsilon,\eta>0$, and choose $M$ from the tail assumption. Take $M'>M$. The function $\tau^{-1}$ is uniformly continuous on $[-\tau(M'),\tau(M')]$. Therefore, there is $\delta\in(0,\tau(M')-\tau(M))$ such that
\[
|s|\leq M,\quad|\tau(s)-\tau(t)|\leq\delta
\quad\Longrightarrow\quad|s-t|\leq\epsilon.
\]
The two assumptions on the left give $|\tau(s)|\leq\tau(M)$ and
\[
|\tau(t)|\leq|\tau(s)|+\delta<\tau(M').
\]

Apply Markov's inequality \cite[proof of Proposition~1.2.6, p.~16]{bakry-gentil-ledoux2014markov} to the nonnegative function $|\tau\circ a_n-\tau\circ b_n|$. For every sufficiently large $n$,
\[
\xi_n(\{|a_n-b_n|>\epsilon\})
\leq\eta+\frac1\delta
d_{1,\xi_n}(\tau\circ a_n,\tau\circ b_n).
\]
Since $\eta>0$ is arbitrary, $a_n-b_n$ converges to $0$ in measure as $n\to\infty$. The definition of the Ky Fan metric gives $\kf^{\xi_n}(a_n,b_n)\to0$ as $n\to\infty$. This completes the proof.
\end{proof}

\subsection{Closedness of the fixed range and finite-feature approximation}
\label{subsec:app-screened-closed-density}

We first prove that $\DK$ is closed. Let $Z_n\in\DK$ and suppose that $\Box(Z_n,Z)\to0$ as $n\to\infty$. Since $\dconc\leq\Box$, \Cref{eq:dconc-coupling} gives, for every $f\in\overline{F_Z}$, couplings $\pi_n\in\T(\mu_Z,\mu_{Z_n})$ and functions $f_n\in\overline{F_{Z_n}}$ such that
\[
\kf^{\pi_n}(f\circ\pr_1,f_n\circ\pr_2)\longrightarrow0
\qquad(n\to\infty).
\]
Every $f_n$ is $K$-valued, so $f_*\mu_Z$ is supported on $K$. Otherwise, a set on which $f$ has positive distance from $K$ would have positive measure and contradict the displayed convergence in measure. The function $f$ is continuous and $\mu_Z$ has full support. Therefore, $f$ takes values in $K$ at every point. Thus $Z\in\DK$, and $\DK$ is closed.

We next prove density of finite-feature quotients. Take $Z\in\DK$ and $\epsilon>0$. Tightness gives a compact set $L\subset Z$ such that
\[
\mu_Z(Z\setminus L)<\epsilon.
\]
The family $\overline{F_Z}|_L$ is uniformly bounded and equicontinuous. The Arzel\`a--Ascoli theorem \cite[7.15, p.~232]{kelley1955general} shows that it is totally bounded in the uniform norm. We may therefore choose a finite set $G\subset\overline{F_Z}$ such that $G|_L$ is a uniform $\epsilon/2$-net of $\overline{F_Z}|_L$.

Use the coupling $(\operatorname{id}_Z,p)_*\mu_Z$ induced by the quotient map $p\colon Z\to Z/G$. The graph over $L$ is closed and has measure greater than $1-\epsilon$. On this graph, every $f\in\overline{F_Z}$ is within uniform distance $\epsilon/2$ of some member of $G$, and every member of $G$ is matched with itself. Therefore, \Cref{eq:box-formula} gives
\[
\Box(Z,Z/G)\leq\epsilon.
\]
The set $G$ is finite and every member is $K$-valued. Thus $Z/G\in\DM_K(N)$ for some $N$, and $\bigcup_{N\geq1}\DM_K(N)$ is dense in $\DK$.

\section{Finite Measurements and Pyramid Limits}
\label{app:pyramid-technical}

This appendix verifies the fixed-range specializations of the cited refinement and reconstruction results. It also supplies the uniform arguments used to characterize weak convergence by finite measurements and by the series metric. We do not repeat the proofs of the cited general results.

\subsection{Refinements that preserve the fixed range}
\label{subsec:app-K-refinement}

We first examine the construction in \cite[Lemma~5.4]{gds2} used for \Cref{lem:K-domination-refinement}. Let $p\colon\bar Y\to Y$ be a domination map. For every $f\in F_Y$, the construction selects $h_f\in\overline{F_{\bar X}}$ that is close to $f\circ p$ on a Box correspondence, puts
\[
G\coloneqq\{h_f\mid f\in F_Y\},
\]
and takes the quotient $X\coloneqq\bar X/G$. Since $\bar X\in\DK$ and $K$ is closed, every member of $\overline{F_{\bar X}}$ is $K$-valued. Therefore, $X\in\DK$. The inequality $\#G\leq\#F_Y$ also shows that the construction does not increase the number of features.

For \Cref{lem:K-common-upper-bound}, apply \cite[Lemma~5.6]{gds2} with the family consisting only of the identity transformation. This gives a precompact sequence $(Z_n)$ with the required domination relations. Pullback along the domination map $\bar Z_n\to Z_n$ embeds $\overline{F_{Z_n}}$ into the $K$-valued family $\overline{F_{\bar Z_n}}$. The domination map has dense image by full support. Continuity of each feature on $Z_n$ then shows that it is $K$-valued at every point. Thus $Z_n\in\DK$.

For \Cref{lem:finite-net-precompactness}, take only the identity transformation in \cite[Lemma~4.12]{gds2} and fix the parameter there to be $1/2$. Its large-measure closed set, uniformly finite net, and cardinality assumptions then agree with the conditions on $L_Z$, $G_Z$, and $N$ in \Cref{lem:finite-net-precompactness}. Since $G_Z\subset\overline{F_Z}$, the quotient $Z/G_Z$ belongs to $\DK$. By \Cref{lem:bounded-measurement-compact}, all these quotients lie in the common compact layer $\DM_K(N)$.

\subsection{Details of convergence through finite measurements}
\label{subsec:app-pyramid-measurements}

We supply the uniform argument in \Cref{lem:K-pyramid-measurements}. Assume weak convergence and fix $N\in\mathbb N$. For each $A\in\DM_K(\mathcal P;N)$, choose $Y_n\in\mathcal P_n$ such that $Y_n\to A$ in the Box distance as $n\to\infty$. Apply \Cref{lem:K-domination-refinement} with $\bar X=Y_n$ and $Y=\bar Y=A$. We obtain $A_n\in\DM_K(\mathcal P_n;N)$ such that
\[
\Box(A_n,A)\leq\Box(Y_n,A).
\]
A finite cover of the compact set $\DM_K(\mathcal P;N)$ by neighborhoods obtained in this way gives
\[
\sup_{A\in\DM_K(\mathcal P;N)}\Box(A,\DM_K(\mathcal P_n;N))\longrightarrow0\qquad(n\to\infty).
\]

If the reverse directed distance did not converge to zero, we could choose a subsequence and $B_n\in\DM_K(\mathcal P_n;N)$ whose distance from $\DM_K(\mathcal P;N)$ is bounded below by a positive constant. By \Cref{lem:bounded-measurement-compact}, a further subsequence satisfies $B_n\to B$ in the Box distance for some $B\in\DM_K(N)$ as $n\to\infty$. The second condition of weak convergence gives $B\in\mathcal P$. Then $B\in\DM_K(\mathcal P;N)$, which is a contradiction.

Conversely, assume Hausdorff convergence of every finite-measurement set. The approximation constructed in the proof of \Cref{lem:K-pyramid-measurements} establishes the first condition of weak convergence. To verify the second, suppose to the contrary that there are $Y\notin\mathcal P$ and a subsequence such that $\Box(Y,\mathcal P_n)\to0$ as $n\to\infty$. Choose $Y_n\in\mathcal P_n$ such that $Y_n\to Y$ in the Box distance as $n\to\infty$. For each $m\in\mathbb N$, choose a finite quotient $B_m\preceq Y$ satisfying $\Box(B_m,Y)<1/m$. The refinement lemma gives finite-feature approximations $B_{m,n}\preceq Y_n$ such that $B_{m,n}\to B_m$ in the Box distance as $n\to\infty$. For fixed $m$, convergence of the measurement sets implies $B_m\in\mathcal P$. Letting $m\to\infty$ and using closedness of $\mathcal P$ gives $Y\in\mathcal P$, contrary to the choice of $Y$.

\subsection{Reconstruction and the series metric}
\label{subsec:app-pyramid-reconstruction}

For \Cref{lem:finite-measurement-proximity}, \cite[Lemma~6.10]{gds2} selects the features on the $Y$ side corresponding to the features of $A\preceq X$ and constructs the quotient $B\preceq Y$ generated by that finite family. Since $Y\in\DK$, the selected features are $K$-valued, and
\[
B\in\DM(Y;N)=\DM_K(\mathcal P_Y;N).
\]
Combining this construction with the finite-feature Box estimate in \cite[Proposition~6.8]{gds2} gives the factor $2N$ in \Cref{lem:finite-measurement-proximity}.

To apply \cite[Proposition~6.24]{gds2} in \Cref{lem:finite-measurement-reconstruction}, take the monoidal family in that proposition to consist only of the identity transformation. Its measurement sets are then exactly the sets $\DM(X;N)$ used here. For $X,X_n\in\DK$, the equality $\DM(X;N)=\DM_K(\mathcal P_X;N)$ established after \Cref{def:finite-measurements} identifies these sets with the $K$-measurement sets of the corresponding principal pyramids. Thus the fixed-range condition introduces no additional truncation.

We finally record the series estimate used in \Cref{thm:K-pyramid-compactification}. Put
\[
H_N(\mathcal P,\mathcal Q)\coloneqq(\Box)_H(\DM_K(\mathcal P;N),\DM_K(\mathcal Q;N)).
\]
For every $N\in\mathbb N$,
\[
H_N(\mathcal P_n,\mathcal P)\leq2N2^N d_{\Pi,K}(\mathcal P_n,\mathcal P),
\]
so convergence in $d_{\Pi,K}$ implies convergence in every coordinate. Conversely, if every coordinate converges, then, for every $M\in\mathbb N$,
\[
d_{\Pi,K}(\mathcal P_n,\mathcal P)\leq\sum_{N=1}^{M}\frac{H_N(\mathcal P_n,\mathcal P)}{2N2^N}+\sum_{N>M}\frac{1}{N2^N}.
\]
First let $n\to\infty$, and then let $M\to\infty$. We obtain $d_{\Pi,K}(\mathcal P_n,\mathcal P)\to0$ as $n\to\infty$.

\section{Construction and Convergence of Pooling}
\label{app:pooling}

This appendix supplies the case analysis and iterative convergence used in \Cref{lem:pairwise-pooling,lem:tau-concavity,thm:pooling-retraction}, together with the finite-support approximation used in \Cref{thm:tau-tensorization}.

\subsection{Monotonicity of two-coordinate pooling}
\label{subsec:app-pairwise-pooling}

Put $r\coloneqq r_{ij}$ and define
\[
H_+(t)\coloneqq a_i\sigma(t+r)+a_j\sigma(t),\qquad H_-(t)\coloneqq a_i\sigma(t)+a_j\sigma(t+r).
\]
Both functions are strictly increasing. Fix the second coordinate $y$ and vary the first coordinate $x$. If $x<y-r$, then the affected output is $(t,t+r)$, where
\[
H_-(t)=a_i\sigma(x)+a_j\sigma(y),
\]
so both output coordinates increase with $x$. For $y-r\leq x\leq y+r$, the output is $(x,y)$. If $x>y+r$, then the affected output is $(t+r,t)$, where
\[
H_+(t)=a_i\sigma(x)+a_j\sigma(y),
\]
and both output coordinates again increase with $x$.

As $x$ approaches $y-r$ from the lower violation region, its parameter approaches $y-r$. As $x$ approaches $y+r$ from the upper violation region, its parameter approaches $y$. Hence the three formulas agree at both boundaries. If the first coordinate is fixed and the second is varied, the corresponding formulas agree at $y=x-r$ and $y=x+r$. Therefore, $P_{ij}^\sigma$ is coordinatewise nondecreasing and continuous.

Now set $\sigma=\tau$. Suppose that $z_i>z_j+r$, and write the affected coordinates of $P_{ij}z$ as $(p+r,p)$. Preservation of the weighted mean gives
\[
a_i\{\tau(z_i)-\tau(p+r)\}=a_j\{\tau(p)-\tau(z_j)\},
\]
where $z_j\leq p\leq z_i-r$. Since the secant slopes of the concave function $T_s$ decrease from left to right,
\[
a_iT_s(\tau(z_i))+a_jT_s(\tau(z_j))\leq a_iT_s(\tau(p+r))+a_jT_s(\tau(p)).
\]
Write the affected coordinates of $P_{ij}(z+s\boldsymbol1)$ as $(\widetilde p+r,\widetilde p)$. The left-hand side is $H_+(\widetilde p)$, and the right-hand side is $H_+(p+s)$. Strict monotonicity of $H_+$ gives $\widetilde p\leq p+s$. The lower violation region is treated with $H_-$, while equality holds in the feasible region after translation. Thus
\[
P_{ij}(z+s\boldsymbol1)\leq P_{ij}z+s\boldsymbol1
\]
in every case.

\subsection{Convergence of cyclic pooling}
\label{subsec:app-pooling-convergence}

Let $(z^{(n)})$ be the iterates in \Cref{thm:pooling-retraction}. For an unordered pair $e$, put
\[
\Delta_e(q)\coloneqq E(q)-E(P_eq).
\]
The map $P_e$ is continuous, and hence so is $\Delta_e$. Each pooling step preserves the weighted mean of the two affected $\tau$ coordinates and moves them strictly closer whenever it changes them. The identity
\[
a_i u^2+a_jv^2=\frac{(a_i u+a_jv)^2}{a_i+a_j}+\frac{a_ia_j}{a_i+a_j}(u-v)^2
\]
applied before and after the step shows that $\Delta_e(q)\geq0$, with equality if and only if $P_eq=q$. Therefore, $E(z^{(n)})$ is nonincreasing and bounded below, and
\[
E(z^{(n)})-E(z^{(n+1)})\longrightarrow0\qquad(n\to\infty).
\]

We prove that $\|z^{(n+1)}-z^{(n)}\|_\infty\to0$ as $n\to\infty$. Otherwise, compactness of the box containing the iterates and finiteness of the set of pairs give a pair $e$, a number $\epsilon>0$, a vector $v\in\R^N$, and a subsequence on which $P_e$ is applied such that
\[
z^{(n_\ell)}\longrightarrow v,\qquad\|P_ez^{(n_\ell)}-z^{(n_\ell)}\|_\infty\geq\epsilon\qquad(\ell\to\infty).
\]
Continuity gives $P_ev\neq v$ and hence $\Delta_e(v)>0$. On the other hand, $\Delta_e(z^{(n_\ell)})\to0$ as $\ell\to\infty$, which is a contradiction.

Let $q$ be a cluster point, and choose a subsequence such that $z^{(n_\ell)}\to q$ as $\ell\to\infty$. Fix a pair $e_k$. For each $\ell$, choose $h_\ell\in\{0,\ldots,m-1\}$ so that the step starting at $z^{(n_\ell+h_\ell)}$ applies $P_{e_k}$. Every one of the at most $m$ intervening step lengths tends to zero as $\ell\to\infty$, so $z^{(n_\ell+h_\ell)}\to q$. The corresponding energy drops also converge to zero. Thus $\Delta_{e_k}(q)=0$ and $P_{e_k}q=q$. Since $e_k$ was arbitrary, $q\in K_W$.

Every $P_e$ fixes $q$ and is nonexpansive for $d_{\tau,a}$. Therefore, $d_{\tau,a}(z^{(n)},q)$ is nonincreasing and tends to zero along the subsequence converging to $q$. It follows that it tends to zero as $n\to\infty$. Since every $a_i$ is positive, $\tau(z_i^{(n)})\to\tau(q_i)$ as $n\to\infty$ for every $i$. The iterates and $q$ lie in one compact box, and $\tau^{-1}$ is uniformly continuous on its image. Hence $z^{(n)}\to q$ as $n\to\infty$.

\subsection{Finite-support approximation of a general factor}
\label{subsec:app-factor-approximation}

For a Borel probability measure $\nu$ on $W$, put
\[
Q(\nu)\coloneqq\sup_F\inf_G\int_W\int_{X\times Y}|\tau(F(x,w))-\tau(G(y,w))|\,d\pi(x,y)\,d\nu(w),
\]
where $F$ and $G$ range over the real-valued $1$-Lipschitz functions on $X\times_1W$ and $Y\times_1W$, respectively. In particular,
\[
Q(\mu_W)=\overrightarrow D_\tau^{\pi\otimes\pi_W^\Delta}(X\times_1W,Y\times_1W).
\]

For a real-valued function $\varphi$ on $W$, write
\[
\operatorname{Lip}(\varphi)\coloneqq\sup_{w\neq w'}\frac{|\varphi(w)-\varphi(w')|}{d_W(w,w')}
\]
for its least Lipschitz constant, with the supremum over the empty set understood as zero. For probability measures $\nu,\nu'$ on $W$, define
\[
\beta(\nu,\nu')\coloneqq\sup\left\{\left|\int_W\varphi\,d\nu-\int_W\varphi\,d\nu'\right|\mathrel{}\middle|\mathrel{}\|\varphi\|_\infty\leq1,\ \operatorname{Lip}(\varphi)\leq1\right\}.
\]
This is a normalization of the Fortet--Mourier metric \cite[Equation~(3), p.~2]{hille-theewis2023boundedlipschitz}. For fixed $F,G$, put
\[
\varphi_{F,G}(w)\coloneqq\int_{X\times Y}|\tau(F(x,w))-\tau(G(y,w))|\,d\pi(x,y).
\]
This function takes values in $[0,1]$ and satisfies
\[
|\varphi_{F,G}(w)-\varphi_{F,G}(w')|\leq d_W(w,w')
\]
for all $w,w'\in W$. To obtain this estimate, split the difference of the two integrands by the triangle inequality and apply the $1$-Lipschitz properties of $F,G$ and the $1/2$-Lipschitz property of $\tau$. The bound is uniform in $F,G$, so taking the infimum over $G$ and the supremum over $F$ gives
\[
|Q(\nu)-Q(\nu')|\leq\beta(\nu,\nu').
\]

We verify density of finitely supported probability measures for $\beta$. Fix $\epsilon>0$. Tightness gives a compact set $L\subset W$ such that $\mu_W(W\setminus L)<\epsilon/4$. Partition $L$ into finitely many Borel sets of radius less than $\epsilon/2$. Move the mass of each set to one representative point and move the mass outside $L$ to one representative point. The resulting finitely supported probability measure $\nu$ satisfies
\[
\left|\int_W\varphi\,d\mu_W-\int_W\varphi\,d\nu\right|<\epsilon
\]
for every $\varphi$ with $\|\varphi\|_\infty\leq1$ and $\operatorname{Lip}(\varphi)\leq1$. We can therefore choose finitely supported probability measures $\nu_n$ such that $\beta(\nu_n,\mu_W)\to0$ as $n\to\infty$.

Let $W_n\coloneqq\supp\nu_n$ carry the inherited metric and the measure $\nu_n$. The McShane--Whitney extension theorem \cite[Proposition~1.1]{esaki-kazukawa-mitsuishi2024invariants} shows that restriction of $1$-Lipschitz functions from $X\times_1W$ and $Y\times_1W$ to $X\times_1W_n$ and $Y\times_1W_n$, respectively, is surjective. Applying \Cref{lem:finite-factor-contraction} to $W_n$ gives
\[
Q(\nu_n)\leq\overrightarrow D_\tau^\pi(X,Y).
\]
Letting $n\to\infty$ and using the Lipschitz estimate for $Q$ proves directed contraction for $W$.

\section*{Acknowledgments}

The author used Claude and GPT-5.6-series Codex models as AI-assisted tools in preparing this manuscript. The author reviewed and revised the mathematical content and takes full responsibility for the final manuscript.

\section*{Statements and Declarations}

\noindent\textbf{Funding.}
No funding was received for this work.

\medskip
\noindent\textbf{Competing Interests.}
The author has no relevant financial or non-financial interests to disclose.

\medskip
\noindent\textbf{Data Availability.}
No datasets were generated or analyzed during this study.

\end{document}